\documentclass[10pt]{article}

\makeatletter
\DeclareRobustCommand{\qed}{%
  \ifmmode 
  \else \leavevmode\unskip\penalty9999 \hbox{}\nobreak\hfill
  \fi
  \quad\hbox{\qedsymbol}}
\newcommand{\openbox}{\leavevmode
  \hbox to.77778em{%
  \hfil\vrule
  \vbox to.675em{\hrule width.6em\vfil\hrule}%
  \vrule\hfil}}
\newcommand{\qedsymbol}{\openbox}
\newenvironment{proof}[1][\proofname]{\par
  \normalfont
  \topsep6\p@\@plus6\p@ \trivlist
  \item[\hskip\labelsep\itshape
    #1.]\ignorespaces
}{%
  \qed\endtrivlist
}
\newcommand{\proofname}{Proof}
\makeatother

\usepackage[utf8]{inputenc}

\usepackage{bm}
\usepackage{float}
\usepackage{placeins}
\usepackage{color}
\usepackage{latexsym}
\usepackage{dsfont}
\usepackage{amssymb}
\usepackage{comment}
\usepackage{graphicx}
\usepackage{amsmath,amsfonts,amssymb,enumerate,theorem,euscript,array,amsfonts,mathrsfs}
\usepackage{hyperref}
\usepackage{appendix}
\usepackage[T1]{fontenc}
\usepackage{babel}
\numberwithin{equation}{section}
\usepackage{bbm}
\usepackage{subfigure}
\usepackage{color}

\usepackage{stmaryrd}

\usepackage[algo2e,ruled,vlined]{algorithm2e} 
 \usepackage{algorithm}
 \usepackage{algorithmicx}
\usepackage{algpseudocode}
\usepackage{ marvosym }
\usepackage{hyperref}
\usepackage{bm}
\usepackage{xcolor, soul}

\usepackage{comment}

\usepackage{tikz}
\usetikzlibrary{fit,matrix,chains,positioning,decorations.pathreplacing,arrows}

\usepackage[a4paper,margin=1in,heightrounded]{geometry}

\usepackage{algpseudocode}
\usepackage{algorithm}

\usepackage[normalem]{ulem}

\def \b1{\bf{1}}

\def \R{\mathbb{R}}

\def \E{\mathbb{E}}
\def \F{\mathbb{F}}

\def \P{\mathbb{P}}

\def\esssup_#1{\underset{#1}{\mathrm{ess\,sup\, }}}

\def\argmin_#1{\underset{#1}{\mathrm{argmin\, }}}
\def\argmax_#1{\underset{#1}{\mathrm{argmax\, }}}

\def\dm#1{\frac{\delta}{\delta m}}

\def \Ac{{\cal A}}
\def \Bc{{\cal B}}

\def \Ec{{\cal E}}
\def \Fc{{\cal F}}

\def \Pc{{\cal P}}

\def \Oc{{\cal O}}

\def\reff#1{{\rm(\ref{#1})}}

\def\ta{{\bf T}}

\def\beqs{\begin{eqnarray*}}
\def\enqs{\end{eqnarray*}}
\def\beq{\begin{eqnarray}}
\def\enq{\end{eqnarray}}

\newtheorem{Theorem}{Theorem}[section]

\newtheorem{Proposition}{Proposition}[section]

\newtheorem{Lemma}{Lemma}[section]

\newtheorem{Remark}{Remark}[section]

\numberwithin{equation}{section}

\title{
Multi-Asset Utility Maximization with Jump Signals}

\author{
  Sigui Brice Dro\footnote{BPCE SA and LPSM, Sorbonne University and
  Universit\'e Paris Cit\'e, \textsf{dro@lpsm.paris}.
  The work of this author is partially supported by
  Agence Nationale de la Recherche (ReLISCoP grant ANR-21-CE40-0001).}
}
\begin{document}

\maketitle

\begin{abstract}
In this paper, we study the portfolio utility maximization problem in a market where the risky assets are driven by correlated Brownian motions and a common independent homogeneous Poisson random measure, and where investment strategies may incorporate jump signals.

Following the same approach as in the one-dimensional case for the exponential utility function \cite{turki2026portfolio}, we first represent the portfolio dynamics as semimartingale processes. We then use martingale optimality principle to derive the corresponding backward stochastic differential equation (BSDE) with jumps to characterize both value function and an optimal strategy in terms of its solution. We subsequently prove the existence and uniqueness of the solution to the related BSDE with jumps. 

We also address the idiosyncratic case i.e where the investor receive heterogeneous jump signals. Finally, we provide numerical illustrations with Gaussian signals for the logarithmic case.
\end{abstract}


\tableofcontents{}
\newpage
\section{Introduction and Motivation}
The aim of this work is to solve Merton problem in a discontinuous framework with jump signals in the multidimensional setting. In this framework, asset prices are driven by a correlated multidimensional Brownian motion and an independent Poisson point process. 
The utility maximization problem is a classic topic in mathematical finance and has been extensively studied since the seminal works of Merton \cite{Merton69, Merton71}. In the continuous case, Cvitanic and Karatzas solve the utility maximization problem under convex portfolio constraints for exponential utility using convex duality techniques. In a related work, El Karoui and Rouge \cite{rouge2000pricing} use BSDEs to compute the value function and the optimal strategy when strategies take values in a convex cone. Later Hu et al. in \cite{hu2005utility} relax this assumption to closed  sets and apply the martingale optimality principle, leading to quadratic BSDEs as studied by Kobylanski \cite{kobylanski}.

In the discontinuous case, Morlais studies the problem for finite activity jumps \cite{MAM08} and later extends the results to the infinite activity case \cite{MAM09}. When incorporating jump signals, Bank and K\"orber analyze the power utility maximization problem using dynamic programming principle in \cite{bank2022merton} for a single asset, and extend their results to a multi-agent framework in \cite{BS25}. Using BSDE techniques, Turki, Dro, and Idris \cite{turki2026portfolio} consider the exponential utility case for an infinite jumps activity.

Our main contribution is to extend these results  considering exponential, power and logarithm utility function to the multidimensional setting with correlated Brownian motion and independent jumps, assuming that the investor receives signals across all assets. We apply the martingale optimality principle to derive a BSDE, for which we establish existence in the degenerate case using results from \cite{MAM08} and \cite{turki2026portfolio}. 

For exponential function in the non-degenerate case, the correlation structure challenges the truncation method used in  \cite{turki2026portfolio}, \cite{MAM08} to get existence of solution. Due to the specific form of the driver of the BSDE, we employ inf-convolution techniques. For power and logarithmic utility case, the driver is characterized by a supremum contrary to the exponential case given by an infimum but remains tractable because the related BSDEs fall into the standard framework which allows to get existence and unicity of solution. Moreover, we deal with the idiosyncratic case when the investor receives heterogeneous jump signals. From a numerical perspective, while \cite{bank2022merton} focuses on power utility and \cite{turki2026portfolio} on exponential utility, we concentrate here on the logarithmic utility  when investor have the same signal on two assets. In particular, for a fixed correlation parameter the more the signal is reliable, the larger the investor's expected utility.\\
The rest of the paper is organized as follows. In Section~\ref{sect2}, we present the probabilistic market model.
In Section~\ref{sect3}, we introduce the class of admissible classical trading strategies together with an extended class of signal-based strategies. We then study the exponential utility maximization problem, derive the associated BSDE with jumps via the martingale optimality principle, and establish the existence and uniqueness of its solution. Section~\ref{sect4} is devoted to the derivation of the corresponding BSDE and the characterization of the optimal strategy for power and logarithmic utility functions under the positive wealth constraint. In Section~\ref{sect6}, we investigate the case where the investor receives different signals from the underlying assets. Finally, Section~\ref{sect7} presents numerical experiments in a two-asset setting.

\section{Market model}\label{sect2}
\subsection{Probabilistic settings}
We consider a complete probability space \((\Omega, \mathcal{A}, \mathbb{P})\) endowed with a $d$ correlated Brownian motion \((W^1_t, ..., W^d_t)_{t \geq 0}\) with a non singular correlation matrix $\Sigma = (\rho _{ij})_{1 \leq i, j \leq d}$ and  an independent  homogeneous Poisson measure $N$ defined on a polish mark space $(E, \mathcal{E})$. We suppose that the compensator of $N$ is \(dt \otimes \nu(de)\), where $\nu$ is a finite positive measure such that
\begin{equation}
\nu(\{0\})~=~0.
\end{equation}
We denote by $\widetilde{N} $ the compensated random measure of  $N$ and we recall that it is given by 
\begin{equation}
\widetilde{N}(dt,de) = N(dt,de)-\nu(de)dt\;.
\end{equation}

We define the  filtration \(\mathbb{F}=(\mathcal{F}_t)_{t \geq 0}\)  as the completion of the natural filtration generated by $W^1, \dots, W^d$ and $N$. $\mathbb{F}$ is therefore given by
\beqs
\mathcal{F}_t = \sigma(W^i_s, N([0,s] \times A) : s \leq t, A \in \mathcal{E}, i =1,...,d)\vee \mathcal{N}\;,\quad t\geq 0\;,
\enqs
where $\mathcal{N}$ stands for the class of $\mathbb{P}$-negligible sets, and $\mathbb{F}$ 
and satisfies the usual conditions (i.e. right-continuity and $\mathcal{F}_0$ contains the $\mathbb{P}$-null sets).
Let \(\mathcal{P} = \sigma((Y_t)_t: Y \text{ is a continuous, } \mathbb{F}\text{-adapted process})\) denote the \(\sigma\)-algebra of predictable processes.\\
We recall the definition of bounded mean oscillation (BMO) martingales and Kazamaki's criterion, which ensures that the stochastic exponential of a BMO martingale is a true martingale. For a comprehensive treatment of this theory, we refer to \cite{kazamaki2006continuous}.

A martingale $M$ is said to belong to the class of BMO martingales if there exists a constant $c > 0$ such that for every $\mathcal{F}$-stopping time $\tau$,
\begin{align}
\operatorname*{ess\,sup}_{\Omega} \, \mathbb{E}^{\mathcal{F}_\tau}\!\left[ \langle M \rangle_T - \langle M \rangle_\tau \,  \right] \leq c^2,
\quad \text{and} \quad
|\Delta M_\tau|^2 \leq c^2.
\end{align}
We shall use in the sequel the following result
\begin{Lemma} (Kazamaki's criterion) \label{Kazamaki_lemma}
Let $\delta > 0$ and let $M$ be a BMO martingale such that $\Delta M_t \geq -1 + \delta$, $\mathbb{P}$-a.s. for all $t \in [0,T]$. Then the Doléans--Dade exponential $\mathcal{E}(M)$ is a true martingale.
\end{Lemma}
\subsection{The financial market}

We follow the same framework as in\cite{BS25}. We consider a frictionless financial market over a time period $[0,T]$  where $T>0$ is a deterministic finite horizon time. We suppose that this financial market is composed by $d + 1 $ assets. The first one is a riskless asset that we assume for simplicity to have a zero interest rate without lost of generality.
The $d$ risky assets  denoted by $S^{i}$ are defined by the following SDE
\begin{equation}
S^{i}_t =  s_{i}+ \int_0^tS^{i}_{u-}[ \kappa_i du + \sigma_i dW^i_u + \int_{E}\eta_i(e) \widetilde{N}(du, de)],\qquad  t\in[0, T]\;
\end{equation} 
for some constants $s_i>0$, $\kappa_i \in \mathbb{R}$, $\sigma_i >0$ , which represent, respectively, the initial value of $S^i$, the drift and the volatility of the asset $S^i$  and by $\theta_i = \frac{\kappa_i}{\sigma_i}$ the market price of risk for $i = 1, ..., d$.
We denote by $\sigma = (\sigma_1, ..., \sigma_d)$ and $\eta = (\eta_1, \dots, \eta_d)$.
The different maps $\eta_i: E\xrightarrow{}\mathbb{R}$ are Borel bounded maps specifying the sensitivity of jumps. Under the previous condition, we have the existence and uniqueness of the processes $S^i$ (see for e.g. \cite[Chap.  III Theorem 2.32]{jacod2013limit}). 
\begin{Remark}
    The use of a common Poisson random measure $N$ reflects the assumption that all assets are exposed to the same market-wide shocks such as macroeconomic announcements, credit events, or liquidity crises, but respond to them heterogeneously, each with its own jump size profile $\eta_i$. This models systematic jump risk.
\end{Remark} 
We define for $e \in E$, the matrix 
\begin{equation}\label{assump}
    \ta(e)  = (\eta_i(e)\eta_j(e))_{1\leq i, j\leq d}
\end{equation} 
 and we recall that it is non negative. In addition, we assume that $\nu(\{\eta_i<-1\}) = 0$, ensuring that each process $S^i$ remains nonnegative.

\section{Exponential utility maximization}\label{sect3}
\subsection{Investment strategies and related wealth}
For $i = 1, ..., d$, let denote by $\pi^i_t$  the wealth invested in the risky asset $S^i$ at time $t\in[0,T]$.
In the classical case, we consider  strategies  $\pi=(\pi^1_t, ..., \pi^d)_{t\in[0,T]}$ that are $\Pc$-measurable processes and satisfy 
\beq\label{cond-Adm-pi}
\sum_{i = 1}^d\int_{0}^{T}|\pi^i_{t}|^2 dt
   <   +\infty\;,\quad  \P-a.s.
\enq
A strategy $\pi=(\pi_t)_{t\in[0,T]}$ that is $\Pc$-measurable and  satisfies \eqref{cond-Adm-pi} is said to be admissible. We denote by $\Ac$ the set of admissible strategies.
For a given initial wealth $x\in \R$ and an investment strategy $\pi=(\pi_t)_{t\in[0,T]}\in \Ac$, we denote by $X^{x,\pi}_t$ the resulting self-financing wealth at time $t\in[0,T]$. The process $X^{x,\pi}$ satisfies the following self-financing dynamics
\beq\label{defDynX}
X^{x,\pi}_t  =  x + \sum_{i = 1}^d \int_0^t \frac{\pi^i_u}{S^i_{u-}} dS^i_u
\;,\quad t\in[0,T]\;.
\enq
We notice that $X^{x,\pi}$ is well defined for $\pi\in \Ac$ according to condition \eqref{cond-Adm-pi}. 

We consider investment strategies in the case where the investor has access to an identical extra information on all the stock $S^i$.  More precisely, we suppose that
the extra information on the stock $S^i$ is given by an impending signal process $G$ defined by
\begin{equation}
G_{t}   =  \int_{0}^{t} \int_{E}\gamma(e) \widetilde{N}(ds, de)\;,\quad t \geq 0\;, 
\end{equation}
where
 $\gamma:E \xrightarrow{}{\mathbb{R}}$ is a Borel map
 satisfying 
 \begin{equation}
\int_E |\gamma(e)|^2 \nu(de)  <  \infty.
 \end{equation}
We also assume that $\gamma(E)$ is a Borel subset of $\R$.
We now consider the set  of strategies including this additional information given by the signal. Namely, we would like to consider strategies 
\beqs
\pi=(\pi_t)_{t\in[0,T]} \text{ that are }\mathcal{P} \vee \sigma(G)\text{-measurable and satisfy  } \eqref{cond-Adm-pi}\;.
\enqs

By \cite[Lemma 2.2]{bank2022merton} a process $\pi$ is $\mathcal{P} \vee \sigma(G)$-measurable if and only if there exists maps $p^i:~\R_+\times\Omega\times\R\rightarrow\R$ which is 
 $\mathcal{P} \otimes \mathcal{B}(\mathbb R)$-measurable, such that
\beqs
\pi^i_t & = & p^i_t(\Delta_t G)\;,\quad t\geq0 ~\P-a.s.
\enqs for $i = 1, ..., d$
where $\Delta_t G=G_t-G_{t-}$ for $t\geq0$.
In particular, we can rewrite the set of such strategies as
\beqs
 \text{processes }(\pi_t=(p^1_t(\Delta G_t), ..., p^d_t(\Delta G_t))_{t\in[0,T]} \text{ satisfying  } \eqref{cond-Adm-pi} \mbox{ with } p^i  \text{ being }\mathcal{P} \otimes \Bc(\R)\text{-measurable} \;.
\enqs

We define the kernel $K:~\R\times \Bc(E)\rightarrow \R $ such that 
\begin{equation}
\nu\big(de \cap \{\gamma \neq 0\}\big) = \int_{\gamma(E) \setminus\{0\}} K(g,de)\mu(dg),  
\end{equation}
 where $\mu$ is the measure image of $\nu $ by $\gamma$, $i.e.$ $\mu = \nu \circ \gamma^{-1}$. To get such a kernel $K$, we consider the measure $M$ defined on $\Bc(\R)\otimes \Bc(\R)$ by
 \begin{equation}
M(A\times B)= \nu\big(\gamma^{-1}(A\setminus \{0\})\cap B\big)
 \end{equation}
 for any $A\in \mathcal{E}$ and $B\in\Bc(\R)$. Applying \cite[Chapter II, Paragraph 1.2]{jacod2013limit}, we get the existence of such a kernel $K$ as the first marginal of $M$ is  $\mu\mathds{1}_{\gamma(\R)\setminus\{0\}}$. Moreover, we have $K(g,E)=K(g,\{\gamma=g\})=1$ for all $g\in \gamma(E)\setminus\{0\}$.
 
Finally, we shall consider the set of admissible strategies with signal $\Ac_{\mathrm{sgn}}$ defined by
\beqs
\Ac_{\mathrm{sgn}} & = & \Big\{\text{Processes }(\pi_t=(p^1_t(\Delta G_t), ..., p^d_t(\Delta G_t)))_{t\in[0,T]} \text{ satisfying  } \eqref{cond-Adm-pi} \mbox{ with } p^1, ..., p^d  \text{ being }\mathcal{P} \otimes \Bc(\R)\text{-measurable}\Big\} \;.
\enqs
For $\pi\in \Ac_{\mathrm{sgn}}$ and for an initial endowment $x\in\R$, the self financing wealth process $ X^{x,\pi}$ defined by $ X_0^{x,\pi}=x$ and 
\begin{equation}\label{defDynX}
           d X_{t}^{x, \pi}  
           = \sum_{i=1}^d \pi^i_t[\kappa_i dt + \sigma_i dW^i_t]  +  \int_{E}\pi^i_t\eta_i(e) \widetilde{N}(dt, de),\quad  t\in [0, T].
\end{equation}
For such a strategy $\pi$, we aim to define a portfolio value process $X^{x,\pi}$
satisfying \eqref{defDynX}. This requires making sense of the stochastic integral
of $\pi$ with respect to $N$. In general, such a stochastic integral is only defined
for $\mathcal{P}$-measurable integrands; however, since $\pi$ possesses additional
measurability, and the measure $\nu$ is finite, the integral reduces to a finite sum
and is therefore well-defined path-wise. It is important to note that for each $i = 1, \ldots, d$, the stochastic integral of
$\pi^i\eta_i$ with respect to $\widetilde{N}$ is not a martingale, but only a semi-martingale. More precisely we have the following decomposition
\begin{align*}
    \int_0^t \int_{E} \pi^i_s\,\eta_i(e)\,\widetilde{N}(ds, de)
    &= \int_0^t\int_{E} p^i_s(\gamma(e))\,\eta_i(e)\,\widetilde{N}(ds,de) \\
    &\quad + \int_0^t\int_{E}
      \bigl(p_s^i(\gamma(e)) - p_s^i(0)\bigr)\eta_i(e)\,\nu(de)\,ds,
    \quad t \geq 0.
\end{align*}
\subsection{The utility maximization problem and related BSDE}

We 
 consider an exponential utility function $U_{\lambda}:~\R\rightarrow\R$ given by 
\begin{equation}
    U_{\lambda}(x) = -\exp(-\lambda x)\;,\quad x\in\R\;,
\end{equation}
where $\lambda$ is a given positive constant which quantify the risk aversion of the investor.

We also consider an $\mathcal{F}_T$-measurable random variable $F$ which can represent the payoff of some financial product. 
We then look for the optimal expected utility for an investor selling the product $F$. This remains to maximize $\E\big[U_{\lambda}(X^{x,\pi}_T - F)\big]$ over the set of admissible strategies. 

We then denote by $\bar \Ac_{\mathrm{sgn}}$  the subset of $\Ac_{\mathrm{sgn}}$ (resp.  $\bar \Ac$  the subset of $\Ac$) representing the set of allowed strategies with signal (resp. without signal).

To define this set, we shall also impose an additional condition 
on $\pi$. More precisely, we fix some constants $\overline \pi, \underline \pi>0$ and we ask $\pi$ to satisfy
\beq
\label{limited-credit-line}
\pi_t & \in & C_d = [-\underline \pi,\overline \pi]^d \quad \mbox{ for all }t\in[0,T]\;,~\P-a.s.
\enq
We observe that under \reff{limited-credit-line}, conditions \eqref{cond-Adm-pi} is satisfied. Therefore, the sets  $\bar \Ac_{\mathrm{sgn}}$ and $\bar \Ac$  are therefore given by
\begin{align*}
\bar{\mathcal{A}}_{\mathrm{sgn}} &= \Bigl\{ 
    \pi = \bigl(p^1_t(\Delta G_t), \ldots, p^d_t(\Delta G_t)\bigr)_{t\in[0,T]} 
    \;\Big|\; 
    p^1, \ldots, p^d \text{ are } \mathcal{P}\otimes\mathcal{B}(\mathbb{R})\text{-measurable}
    \text{ and satisfy } \eqref{limited-credit-line}
\Bigr\},\\[6pt]
\bar{\mathcal{A}} &= \Bigl\{ 
    \pi = (\pi_t)_{t\in[0,T]} 
    \;\Big|\; 
    \pi \text{ is } \mathcal{P}\text{-measurable and satisfies } \eqref{limited-credit-line}
\Bigr\}.
\end{align*}
We then define  
the value functions $V_{\mathrm{sgn}}$ and $V$ by
\beq\label{defV(x)}
    V_{\mathrm{sgn}}(x) & = & \sup_{\pi \in \bar\Ac_{\mathrm{sgn}}}\E\big[U_{\lambda}(X^{x,\pi}_T - F)\big]\;,
\enq 
and
\beq\label{defVw(x)}
    V(x) & = & \sup_{\pi \in \bar\Ac}\E\big[U_{\lambda}(X^{x,\pi}_T - F)\big]\;,
\enq 
for $x\in\mathbb{R}$. Let recall that the study of the utility maximization problem with signal \reff{defV(x)} is a generalization of the one dimensional case of \cite{MAM08} for the study of the utility maximization problem without signal  \reff{defVw(x)}.

For a given initial endowment $x\in\mathbb{R}$, we shall also look for  a related optimal strategy $\pi^*\in \bar{\mathcal{A}}_{\mathrm{sgn}}$ such that 
\begin{equation}\label{defOI}
    V_{\mathrm{sgn}}(x) = \E\big[U_{\lambda}(X^{x,\pi^*}_T - F)\big].
\end{equation}

To solve problem \eqref{defV(x)}-\eqref{defOI}, we follow the martingale optimality approach presented in \cite{hu2005utility}. This approach consists in the characterization of the optimality by a martingale criterion as shown by the following result.
\begin{Proposition}\label{MgOptPr}
    Suppose that there exists a family of processes $(R^\pi)_{\pi\in \bar \Ac_{\mathrm{sgn}}}$ satisfying the following conditions. 
\begin{enumerate}[(i)]
\item $R_T^\pi= U_{\lambda}(X^{x,\pi}_T-F)$ for any $\pi\in \bar \Ac_{\mathrm{sgn}}$.
\item There is some constant $R_0$ such that $R^\pi_0=R_0$ for any $\pi\in\bar \Ac_{\mathrm{sgn}}$.
\item The process $(R^\pi_t)_{t\in[0,T]}$ is a supermartingale for any $\pi\in \bar \Ac_{\mathrm{sgn}}$.
\item There is some $\pi^*\in\bar \Ac_{\mathrm{sgn}}$ such that $(R^{\pi^*}_t)_{t\in[0,T]}$ is a martingale.
\end{enumerate}
Then $R_0=V_{\mathrm{sgn}}(x)$ and $\pi^*$ is an optimal investment strategy.
\end{Proposition}
\begin{proof}
    Fix some $\pi\in \bar \Ac_{\mathrm{sgn}}$. We then have from (i), (ii) and (iii)
    \beqs
\E[U_{\lambda}(X^{x,\pi}_T-F)] & = & \E[R^{\pi}_T]~\leq~R^{\pi}_0~=~R_0\;.
    \enqs
    Therefore, we get
    \beqs
    \sup_{\pi\in\bar \Ac_{\mathrm{sgn}}}\E\big[U_{\lambda}(X^{x,\pi}_T-F)\big]~\leq~R_0\;.
    \enqs
    Using (iv), we get
    \beqs
    \E\big[U_{\lambda}(X^{x,\pi^*}_T-F)\big]~=~R_0\;.
    \enqs
    Hence $\pi^*$ is optimal and $V_{\mathrm{sgn}}(x)=R_0$\;.
\end{proof}
In the sequel, we construct  a solution to problem \eqref{defV(x)}-\eqref{defOI} based on a family $(R^\pi)_{\pi\in \bar \Ac_{\mathrm{sgn}}}$ satisfying Proposition \ref{MgOptPr}. 

For this purpose, we use the theory of Backward SDEs with jumps (BSDEJ) for short. We next present the notion of BSDEJ and we refer to \cite{delong13} for a detailed presentation of this theory.  
We introduce the following spaces of processes and maps. 


\begin{itemize}
    \item \( S^{\infty} = \left\{ \F\text{-adapted c\`adl\`ag processes } Y  \text{ valued in $\R$ such that } \esssup_{t \in [0,T]} |Y_t| < \infty \right\} \).
    
    \item \( L^2(W) = \left\{\Pc\text{-measurable processes } Z \text{ valued in $\R^d$ such that } \mathbb{E} \left[\int_0^T |Z_t|^2 \, dt \right] < \infty \right\} \).
    
    \item \( L^2(\widetilde{N}) = \left\{\mathcal{P} \otimes \mathcal{B}(E)\text{-measurable processes } U \text{ valued in $\R$ such that } \mathbb{E} \left[\int_0^T \int_E U_t^2(e) \, \nu(de) \, dt \right] < \infty \right\} \).
    
    \item $L^2(\nu) = \{ u : E \xrightarrow{} \mathbb{R} \mbox{ Borel, } \int_{E}u^2(e)\nu(de) <\infty \}$. 
    \item $L^\infty(\nu) = \{ u : E \xrightarrow{} \mathbb{R} \mbox{ Borel and bounded}\}$.
\end{itemize}

Next, we fix a terminal condition given an $\Fc_T$-measurable random variable $\xi$ and a function $f:~\Omega\times [0,T]\times \R\times \R^d\times L^2(\nu)$ that is assumed to be $\Pc\otimes\Bc(\R)\otimes\Bc(\R^d)\otimes\Bc(L^2(\nu))$-measurable.

A solution to the BSDEJ with parameters $(\xi,f)$ is a triple of processes $(Y, Z, U)\in S^\infty \times L^2(W) \times L^2(\widetilde{N})$ satisfying 
\begin{equation}
Y_t  =  \xi + \int_{t}^{T} f(s, Y_s, Z_s, U_s) \, ds 
- \int_{t}^{T} Z_s \, dW_s 
- \int_{t}^{T} \int_{E} U_s(e) \, \widetilde{N}(ds,de), \quad t \in [0, T]\;.\label{BSDEJ}
\end{equation}

Following the approach initiated by \cite{hu2005utility} in the Brownian framework and then developed by \cite{MAM08} in the mixed Brownian-Poisson case, we look for the family $(R^\pi_t)_{t\in[0,T]}^{\pi\in \bar\Ac_{\mathrm{sgn}}}$ satisfying Proposition \ref{MgOptPr} under the following form
\begin{equation}
R^\pi_t  =  U_{\lambda}\big(X^{\pi}_t-Y_t\big)\;,\quad t\in[0,T]\;,
\end{equation}
where $(Y_t)_{t\in[0,T]}$ satisfies a BSDE with some parameters $(\xi,f)$. We next face two questions. The first one consists in finding appropriate $\xi$ and $f$ to satisfy conditions (i) to (iv) of Proposition \ref{MgOptPr}. Given, this appropriate coefficient, the second question is to prove existence, and possibly uniqueness, of solution to the considered BSDEJs. 
\vspace{2mm}

For $u \in L^2(\nu)$ and $\lambda > 0$, we define the following non-negative functional:
\begin{equation}
|u|_{\lambda} := \int_E h_{\lambda}(u(e))\, \nu(de),
\end{equation} where $h_\lambda$ is the function defined by $h_\lambda(x) = \frac{e^{\lambda x} - \lambda x - 1}{\lambda}$ for all $x\in \mathbb{R}$.




Following the approach described in the previous section, we construct the family of processes $(R^\pi)_{\pi\in \bar \Ac_{\mathrm{sgn}}}$ introduced in Proposition \ref{MgOptPr}, by setting
\begin{equation}\label{defRpi}
R^{\pi}_{t} = -\exp\left(-\lambda\left(X^{x,\pi}_{t} - Y_t\right)\right), \quad t \in [0, T], \; \pi \in \bar \Ac_{\mathrm{sgn}},
\end{equation}
where $(Y, Z, U)$ is a solution to the BSDEJ \eqref{BSDEJ}. From condition (i) of Proposition \ref{MgOptPr}, we need to have $\xi=F$. Therefore $(Y, Z, U)$ satisfies 
\begin{equation}\label{BSDEexp}
Y_t = F + \int_{t}^{T} f(s, Z_s, U_s) \, ds - \sum_{i =1}^d\int_{t}^{T} Z^i_s \, dW^i_s - \int_{t}^{T} \int_{\R} U_s(e) \, \widetilde{N}(ds, de), \quad t \in [0, T].   
\end{equation}


    
    


In order to compute $f$, we apply Ito formula to $R^{\pi}$. Define the process $L^\pi$ by $L^\pi_t = X^{x, \pi}_t - Y_t$, we have 
\beqs
dL^\pi_t &= &  \left[ \sum_{i=1}^d p^i_t(0) (\kappa_i - \int_E \eta_i(e)\nu(de)) + \int_E U_t(e)\nu(de) + f(t, Z_t, U_t) \right] dt  +  \sum_{i = 1}^d \left[ p^i_t(0) \sigma_i - Z^i_t\right] dW^i_t 
\\
& &  + \int_{E} \Big(\sum_{i=1}^d p^i(\gamma(e))_t\eta_i(e) - U_t(e) \Big)N(dt, de) 
\;,\quad t\in[0,T]\;.
\enqs
Due to the correlation, the quadratic variation of the continuous part is given by
\begin{equation*}
d\langle L^{\pi,c} \rangle_t 
= \sum_{i=1}^d \left(p_t^i(0)\,\sigma_i - Z_t^i\right)^2 \, dt 
+ \sum_{\substack{i,j=1 \\ i \neq j}}^d 
\left(p_t^i(0)\,\sigma_i - Z_t^i\right)
\left(p_t^j(0)\,\sigma_j - Z_t^j\right)
\rho_{ij} \, dt.
\end{equation*}


Let denote by $\tilde{\kappa_i} = \kappa_i - \int_E\eta_i(e)\nu(de)$, finally applying  Itô's formula to $R^\pi=-\exp(-\lambda L^\pi)$ gives

\begin{align}
dR_t^\pi 
&= R_{t-}^\pi \Bigg[
-\lambda \Bigg(
\sum_{i=1}^d p_t^i(0)\,\tilde{\kappa_i} + \int_E U_t(e)\nu(de)
+ f(t,Z_t,U_t)  
- \int_{E} \Big( \sum_{i=1}^d p_t^i(\gamma(e))\,\eta_i(e) - U_t(e) \Big)\,\nu(de)
\Bigg) dt \nonumber \\
&\qquad\quad 
+ \frac{1}{2}\lambda^2 \, d\langle L^c \rangle_t
- \lambda \sum_{i=1}^d \big( p_t^i(0)\,\sigma_i - Z_t^i \big)\, dW_t^i \nonumber \\
&\qquad\quad 
+ \int_{E} \Big(
\exp\Big( -\lambda \big( \sum_{i=1}^d p_t^i(\gamma(e))\,\eta_i(e) - U_t(e) \big) \Big) - 1
\Big)\, N(dt,de)
\Bigg],
\quad t \in [0,T].
\end{align}

Therefore, we get
\beqs
dR_t^\pi &= & R_{t-}^\pi \Big[ \Lambda_t dt +\sum_{i=1}^d \Gamma^i_t dW^i_t+\int_E \Theta_t(e) N(de,dt)\Big]
\enqs
with
\beqs
\Gamma^i_t & = & - \lambda \left(p^i_t(0)\sigma_i - Z^i_t\right) , \mbox{ for } i = 1,..., d \\
\Theta_t(e) & = &  \exp\left(-\lambda(\sum_{i=1}^d p^i_t(\gamma(e))\eta_i(e) - U_t(e))\right) - 1 \\
\Lambda_t & = & -\lambda \Bigg(
\sum_{i=1}^d p_t^i(0)\,\tilde{\kappa_i}
+ f(t,Y_t,Z_t)  
+ \int_{E} U_t(e) \nu(de)
\Bigg) dt
+ \frac{1}{2}\lambda^2 \, d\langle L^{\pi,c} \rangle_t  
\enqs
for $t\in[0,T]$ and $e\in E  $.
%

To satisfy condition (iii) of Proposition \ref{MgOptPr}, we impose to the  drift part of $R^\pi_t$ to be negative, which gives
\beqs
-\lambda \Bigg(
\sum_{i=1}^d p_t^i(0)\,\tilde{\kappa_i}
+ f(t,Y_t,Z_t)  
+ \int_{E} U_t(e) \nu(de)
\Bigg) dt
+ \frac{1}{2}\lambda^2 \, d\langle L^{\pi,c} \rangle_t   +\int_{E} \Theta_t(e) \nu(de)
 & \geq  & 0.
\enqs
Using the identity 
\begin{equation*}
    \frac{\lambda}{2}(p^i(0)\sigma_i - Z^i)^2 - p^i \kappa_i = \frac{\lambda}{2}[p^i(0) \sigma_i - (Z^i + \frac{\theta_i}{\lambda})]^2 - Z^i \theta_i - \frac{(\theta_i )^2}{2 \lambda}
\end{equation*}
we have
\begin{align*}
 -\lambda f(t, Z_t, U_t) 
&+ \frac{1}{2} \lambda^2 \left( \sum_{i=1}^d \left[ p^i_t(0)\sigma_i - \left( Z^i_t + \frac{\theta_i}{\lambda} \right) \right]^2 + \sum_{\substack{i,j=1 \\ i \neq j}}^d 
\left( p_t^i(0)\,\sigma_i - Z_t^i \right)
\left( p_t^j(0)\,\sigma_j - Z_t^j \right)
\rho_{ij} \right) \\
&- \sum_{i=1}^d \left( \lambda Z^i_t \theta_i + \frac{(\theta_i)^2}{2} \right)
+ \int_{E} \lambda [h_{\lambda}\!\left( U_t(e) - \sum_{i=1}^d p^i_t(\gamma(e))\, \eta_i(e) \right) -  \sum_{i=1}^d p^i_t(\gamma(e))\eta_i(e)]
\nu(de) \geq 0\,.
\end{align*}

Using the disintegration formulation on the set $\{\gamma \neq 0\}$, we can rewrite the previous inequality as
\begin{align*}
-\lambda f(t, Z_t, U_t) 
&+ \frac{1}{2}\lambda^2 \Bigg(
    \sum_{i=1}^d \left[ p_t^i(0)\sigma_i - \left( Z_t^i + \frac{\theta_i}{\lambda} \right) \right]^2 
    + \sum_{\substack{i,j=1 \\ i \neq j}}^d 
    \left( p_t^i(0)\sigma^i - Z_t^i \right)
    \left( p_t^j(0)\sigma^j - Z_t^j \right)
    \rho_{ij}
\Bigg) \\
&- \sum_{i=1}^d \left( \lambda Z_t^i \theta_i - \frac{\theta_i^2}{2} \right) + \int_{\gamma =0} 
    \lambda [h_{\lambda}\!\left( U_t(e) - \sum_{i=1}^d p_t^i(0)\,\eta_i(e) \right)\, - \sum_{i=1}^d p^i_t(\gamma(e))\eta_i(e))]\nu(de) \\
&+ \int_{\gamma(E)\setminus\{0\}} 
    \int_{\gamma = g} 
    \lambda [h_{\lambda}\!\left( U_t(e) - \sum_{i=1}^d p_t^i(\gamma(e))\,\eta_i(e) \right)
    \, - \sum_{i=1}^d p^i_t(\gamma(e))\eta_i(e))]K(g,de)\, \mu(dg) \geq 0.
\end{align*}
To satisfy the previous inequality, 
we define the driver $f$ by
\begin{align}
f(z, u) =  \inf_{p \in C_d} f^1(z,u,p)
\label{defDriverfexp}
 + \int_{\gamma(E) \setminus \{0\}} \inf_{p \in C_d} f^2(g,u,p)
\mu(dg) 
 - \sum_{i=1}^d \left( z^i \theta_i + \frac{\theta_i^2}{2\lambda} \right)
\end{align}
with
\begin{align}
f^1(z, u, p) &= \frac{1}{2} \lambda \left[ \sum_{i=1}^d \left( p^i\sigma_i - \left( z^i + \frac{\theta_i}{\lambda} \right) \right)^2 
+ \sum_{\substack{i,j=1 \\ i \neq j}}^d 
\left( p^i\sigma_i - z^i \right)
\left( p^j\sigma_j - z^j \right)
\rho_{ij} \right] \notag \\
&\quad  + \int_{\gamma = 0} 
[h_{\lambda}\!\left( u(e) - \sum_{i=1}^d p^i\,\eta_i(e) \right) -  \sum_{i=1}^d p^i\eta_i(e)]\nu(de) \notag \\
&\quad 
\end{align}
and
\beqs
f^2( g,u,p) & = & 
\int_{\gamma = g} 
    [h_{\lambda}\!\left( u(e) - \sum_{i=1}^d p^i\,\eta_i(e) \right) -  \sum_{i=1}^d p^i\eta_i(e)] K(g,de)\
\enqs
for $z\in\R^d$, $u\in (L^2\cap L^\infty)(\nu)$, $g\in\gamma(E)$ and $p\in\R^d$.\\
We notice that $f$ does not depend on the component $Y$ which is a main feature of utility maximization problem. \\
    
 
\subsection{Characterization of the optimal strategies}
\begin{Theorem}
    Suppose that the Backward SDE \eqref{BSDEJ} with parameters $\xi=F$ and $f$ given by \eqref{defDriverfexp} admits a solution $(Y,Z,U)\in S^\infty \times L^2(W) \times L^2(\widetilde{N})$. Then, there exists a Borel map 
   $p^*:~\R\times\R^d\times  L^2(\nu)\rightarrow\R^d$  satisfying
\begin{equation} \label{mindriveratpistar}
f(z, u)  =  f^1(z,u,p^*(0,z,u))+
\int_{\gamma(E) \setminus \{0\}}  f^2(g,u,p^*(g,z,u)) 
\mu(dg)
- \sum_{i=1}^d \left( z^i \theta_i + \frac{(\theta_i)^2}{2\lambda} \right)
\end{equation}
and
\beqs
    p^*(g,z,u) \in   C_d 
\enqs
for all $g,z,u\in\gamma(E)\times\R ^d\times (L^2\cap L^\infty)(\nu)$, and we have $V(x)=U_{\lambda}(x-Y_0)$ and $\pi^*=(p^*(\Delta G_t,Z_t,U_t))_{t\in[0,T]}$ is an optimal strategy in $\bar \Ac_{\mathrm{sgn}}$.
\end{Theorem}

\begin{proof}\label{iden_signal_case}
Let $(Y,Z,U)\in S^\infty \times L^2(W) \times L^2(\widetilde{N})$  be solution to BSDEJ \eqref{BSDEexp} with $f$ given by \eqref{defDriverfexp} and consider the family $(R^\pi)_\pi$ defined by \eqref{defRpi}. We prove that the conditions of Proposition \ref{MgOptPr} are satisfied.

The family $(R^\pi)_\pi$  satisfies (i) and (ii) of Proposition \ref{MgOptPr} as $R^\pi_T=U_{\lambda}(X_T^{x,\pi}-Y_T)=U_{\lambda}(X_T^{x,\pi}-F)$ and $R_0^\pi=U_{\lambda}(X_0^{x,\pi}-Y_0)=U_{\lambda}(x-Y_0)$ for any $\pi\in \bar \Ac_{\mathrm{sgn}}$.  From the definition of $f$ and the computations made above, we observe that $R^\pi$  is a nonpositive local supermartingale. We use lemma 1 in \cite{MAM08} which holds due to boundness of the parameters of the model and finiteness of the Levy measure $\nu$ to obtain uniform integrability of $R^\pi$.
Hence by dominated convergence theorem, it is a supermatingale and (iii) of  Proposition \ref{MgOptPr} is satisfied. 

Suppose there exists a Borel function $p^*$ satisfying 
\eqref{mindriveratpistar}. We show that condition \textit{(iv)} is 
satisfied with $\pi^*$. For $\pi = \pi^*$, $(R^{\pi^*}_t)$ is a martingale because it is a local martingale uniformly integrable. 

It remains to show that there exists a Borel function $p^*$ satisfying \reff{mindriveratpistar}. 

Using the dominated convergence theorem, we get that $f^1$ and $f^2$ are twice continuously differentiable w.r.t. $p$ on $\gamma(E)\times \R^d\times L^2(\nu)$ and we have Hessian matrix
\begin{align}
D^2_{p} f^1(z,u,p)
&= \lambda \Bigg(
\operatorname{diag}(\sigma)
+ \int_{\gamma = 0} \ta(e)\,\exp\big(\lambda(u(e) - \langle p, \eta(e) \rangle\big)\,\nu(de)
\Bigg) > 0, \\
D^2_{p} f^2(g,u,p)
&= \int_{\gamma = g} \ta(e)\,\exp\big(\lambda(u(e) - \langle p, \eta(e) \rangle\big)\,K(g,de) \ge 0,
\end{align}
for all $(g,z,u,p)\in \gamma(E)\times \mathbb{R}^d \times (L^2\cap L^\infty)(\nu)\times \mathbb{R}^d$.

This shows that  the functions $p\mapsto f^1(z,u,p)$ and $p\mapsto f^2(g,u,p)$ are respectively strictly convex and convex for any $z\in\R$, $u\in L^2(\nu)$ and $g\in \gamma(E)$. Therefore there exits a function $p^*$ satisfying \reff{mindriveratpistar}.\\We are going to prove by induction on $d$ that the set containing $p^*$ 
is a Borel set when the signal is zero; the complementary case is treat similarly.

When $d=1$, this property is satisfied (cf.\ the proof of the theorem 5.1 in \cite{turki2026portfolio}).

Let $d$ be such that the property is true and denote by $\mathcal{P}^d$ 
the containing set. Let us denote by 
\begin{equation}
\tilde{p}(0,z,u) = \big(p^*_1(0,z,u),\ldots,p^*_d(0,z,u),\, 
p^*_{d+1}(0,z,u)\big)
\end{equation}
the optimal strategy in dimension $d+1$.

By optimality criteria $p^*_{d+1}(0,z,u)$ belongs to the set $A_1 \cup B_1 \cup C_1$, where
\begin{align*}
A_1 &= \left\{(0,z,u,p) ~:~ (z,u,p)\in \mathbb{R}^{d+1}\times 
\mathcal{O}_1 \times \mathbb{R} \mbox{ and } 
\frac{\partial f^1}{\partial p_{d+1}}(z,u,p)=0 \right\}, \\[6pt]
B_1 &= \left\{(0,z,u,-\underline{\pi}) ~:~ (z,u)\in \mathbb{R}^{d+1}
\times \big(L^2(\nu)\setminus\mathcal{O}_1\big) 
\mbox{ and } 
\frac{\partial f^1}{\partial p_{d+1}}(z,u,-\underline{\pi})\geq 0 
\right\}, \\[6pt]
C_1 &= \left\{(0,z,u,\overline{\pi}) ~:~ (z,u)\in \mathbb{R}^{d+1}
\times \big(L^2(\nu)\setminus\mathcal{O}_1\big) 
\mbox{ and } 
\frac{\partial f^1}{\partial p_{d+1}}(z,u,-\underline{\pi}) < 0 
\right\},
\end{align*}
with
\begin{align*}
\mathcal{O}_1 &= \left\{ (z,u)\in \R^{d+1}\times 
L^2(\nu)\ ~:~ 
\frac{\partial f^1}{\partial p_{d+1}}(z,u,-\underline{\pi})
\cdot
\frac{\partial f^1}{\partial p_{d+1}}(z,u,\overline{\pi}) 
< 0 \right\}.
\end{align*}

We then notice that the sets $A_1$, $B_1$, $C_1$ are Borel measurable 
since $\gamma(E)$  is Borel measurable. Hence the set
\begin{equation}
\big\{(0,z,u,p^*(0,z,u)) ~:~ (z,u)\in \mathbb{R}^{d+1}\times (L^2\cap L^\infty)(\nu)\big\}
\end{equation}
is Borel measurable as product of the borelian Borelian space $\mathcal{P}^d \times (A_1 \cup B_1 \cup C_1) $. By induction, we obtain the Borel measurability 
of the optimal strategy when there is no signal.

\end{proof}
\subsection{Existence and Uniqueness of solution to the BSDEJ}\label{existence_solution}

Before proceeding to the proof of existence, we give some useful properties of the driver of \eqref{defDriverfexp} by the following lemma.

\begin{Lemma}\label{lem-ref-f}
The function $f$ satisfies the following properties:
\begin{enumerate}[(i)]

\item
\begin{align*}
- \overline{\pi} \sum_{i=1}^d (\sigma_i|\theta_i| + \int_E|\eta_i(e)|\nu(de))
\;\leq\; f(z,u)
&\leq  \frac{\lambda}{2} z^T \Sigma z + |u|_\lambda
\end{align*}
for all $z \in \mathbb{R}^d$ and $u \in (L^2\cap L^\infty)(\nu)$.

\item There exist constants $C, \kappa >0$ depending on the dimension $d$ such that
\begin{equation}\label{locLipfz}
\left| f(z, u) - f(z', u) \right|
\leq C \bigl( \kappa + |z| + |z'| \bigr)\, |z - z'|
\end{equation}
for all $z, z' \in \mathbb{R}^d$ and $u \in (L^2\cap L^\infty)(\nu)$.

\item
There exist a bounded progressively measurable process $k$ and a constant $\theta$ such that
\begin{equation*}
\left| 
f(z,u) - f(z,u')
- \int_{E} k(u,u')(e)\,\bigl(u(e) - u'(e)\bigr)\,\nu(de)
\right|
\leq \theta \bigl( \|u\|_{L^2(\nu)} + \|u'\|_{L^2(\nu)} \bigr)\, \|u - u'\|_{L^2(\nu)}
\end{equation*}
and 
\begin{equation*}
    f(z,u) - f(z,u')
\leq \int_{E} k(u,u')(e)\,\bigl(u(e) - u'(e)\bigr)\,\nu(de) 
\end{equation*}
 for all $z \in \mathbb{R}^d$ and $u, u' \in (L^2\cap L^\infty)(\nu)$.
 \item Moreover, there exist two constants $\bar C,\underline C>0$ such that
\beqs
\big(-1 + \underline{C}
\big)  &\leq & k(u, u')(e) ~~\leq~~\bar {C}
\enqs

\end{enumerate}
\end{Lemma}

\begin{proof}
i) Since $0_{\mathbb{R}^d} \in C^d$, we have
\begin{equation*}
f(z, u) \leq \frac{\lambda}{2} \left[
\sum_{i=1}^d \left(z^i + \frac{\theta_i}{\lambda}\right)^2 
+ \sum_{\substack{i,j=1 \\ i \neq j}}^d z^i z^j \rho_{ij}
\right]
+ \int_{\mathbb{R}} h_{\lambda}(u(e))\,\nu(de)
- \sum_{i=1}^d \left( z^i \theta_i + \frac{(\theta_i)^2}{2\lambda} \right).
\end{equation*}
Expanding the square terms yields the desired upper bound for $f$.

For the lower bound, define $\tilde{z}_p^i = p^i \sigma_i - z^i$ for $i = 1, \dots, d$. Using the positivity of $h_\lambda$ and expanding the square terms, we obtain
\begin{equation*}
f(z, u, p) \geq  (\tilde{z}_p)^\top \Sigma \tilde{z}_p 
- \sum_{i=1}^d p^i (\sigma_i \theta_i + \int_E\eta_i(e)\nu(de))
+ \sum_{i=1}^d \theta_i z^i 
+ \frac{1}{2\lambda} \sum_{i=1}^d \theta_i^2.
\end{equation*}
Since $\Sigma$ is a correlation matrix, it is positive semi-definite, and thus the first term is nonnegative. Taking the maximal value of $p$ 
we recover the stated lower bound.

\item[(ii)] Let $z, z' \in \mathbb{R}^d$ and $u \in (L^2\cap L^\infty)(\nu)$. Then
\begin{equation*}
f(z,u) - f(z',u) 
\leq \inf_{p \in C_d} f^1(z,u,p) - \inf_{p \in C_d} f^1(z',u,p) 
\leq \sup_{p \in C_d} \bigl( f^1(z,u,p) - f^1(z',u,p) \bigr).
\end{equation*}
Moreover,
\begin{align*}
z^\top \Sigma z - (z')^\top \Sigma z'
&= z^\top \Sigma (z - z') + (z' - z)^\top \Sigma z' \\
&\leq \bigl( |z| + |z'| \bigr)\, |z - z'|,
\end{align*}
by the Cauchy--Schwarz inequality and the fact that the eigenvalues of $\Sigma$ are bounded by $1$. 
Setting $\kappa = 2\overline{\pi}\,|\sigma|_{\max}$, we deduce the desired result.

\medskip

\item[(iii)] We are going to give estimate on $\inf f^1$, the same holds for $\inf f^2$ and we conclude with disintegration formula.

Let $z \in \mathbb{R}^d$ and $u, u' \in (L^2\cap L^\infty)(\nu)$ with $\|u\|_{\infty}, \|u'\|_{\infty} \leq K$ for some $K >0$.
To ease notations, we introduce the scalar product $\langle p, \eta(e) \rangle := \sum_{i=1}^d p^i \eta^i(e)$.

Let $p \in C_d$, for each $e \in E$, a second-order Taylor expansion yields, for some $\lambda = \lambda(e) \in (0,1)$,
\begin{align*}
&h_{\lambda}\bigl(u(e) - \langle p, \eta(e) \rangle\bigr)
- h_{\lambda}\bigl(u'(e) - \langle p, \eta(e) \rangle\bigr) \\
&= h'_{\lambda}\bigl(u'(e) - \langle p, \eta(e) \rangle\bigr)\,(u(e)-u'(e)) \\
&\quad + \frac{1}{2} h''_{\lambda}\bigl(\lambda(e) u(e) + (1-\lambda(e))u'(e) - \langle p, \eta(e) \rangle\bigr)\,(u(e)-u'(e))^2.
\end{align*}

Since $p$ and $\eta$ are bounded and $\|u\|_{\infty}, \|u'\|_{\infty} \leq K$, there exists a constant $M_K > 0$ such that
\[
\left| h''_{\lambda}(\cdot) \right| \leq M_K.
\]

Integrating over $\{\gamma =0\}$ and taking the supremum over $p \in C^d$, we obtain
\begin{align*}
&\inf_{p \in C_d} f^1(z,u,p) - \inf_{p \in C_d} f^1(z,u',p) \\
&\leq \int_{\gamma=0} 
\sup_{p \in C_d} h'_{\lambda}\bigl(u'(e) - \langle p, \eta(e) \rangle\bigr)\,(u(e)-u'(e))\,\nu(de) \\
&\quad + \frac{M_K}{2} \int_{\gamma=0} (u(e)-u'(e))^2\,\nu(de).
\end{align*}

Define
\begin{align*}
k(u,u')(e)
&:= \sup_{p \in C_d} h'_{\lambda}\bigl(u(e) - \langle p, \eta(e) \rangle\bigr)\, \mathbf{1}_{\{u(e) \geq u'(e)\}} \\
&\quad + \inf_{p \in C_d} h'_{\lambda}\bigl(u(e) - \langle p, \eta(e) \rangle\bigr)\, \mathbf{1}_{\{u(e) < u'(e)\}}.
\end{align*}

By construction, $k(u,u')$ is measurable as a supremum/infimum of measurable functions, and it satisfies iv) due to the boundedness of $p$, $\eta$ and $u'$.

Proceeding similarly on $\{\gamma \neq 0\}$, we extend the estimate to the whole space. Finally, applying the Cauchy--Schwarz and triangle inequalities yields the desired bound with a constant $\theta > 0$.
\end{proof}

We turn to the study of a solution to the
Backward SDE \eqref{BSDEJ} with parameters $\xi=F$ and $f$ given by \eqref{defDriverfexp}.\\ We state the following theorem:
\begin{Theorem}
Assume that $F$ is bounded. Then the BSDE \eqref{BSDEJ} admits a unique solution $(Y,Z,U) \in S^\infty \times L^2(W) \times L^2(\widetilde{N})$.
\end{Theorem}

\begin{proof}

We deal with the degenerate and the non-degenerate case.
In the first case where $\rho_{ij} = 0$ for $i \neq j \in \{1, \ldots, d\}$, we introduce the following truncation functions for $m \geq 1$, 
the function $\phi_m$ is defined by
\begin{align*}
\phi_m(x) &= (x + m + 1)\,\mathbf{1}_{(-(m+1),\,-m)}(x) 
           + \mathbf{1}_{[-m,\,m]}(x) 
           + (m + 1 - x)\,\mathbf{1}_{(m,\,m+1]}(x).
\end{align*}
We introduce the following sequence approximation of the driver $f$ with 
\begin{equation*}
    f_m(z,u) = \inf_{p \in C_d} f^1_m(z,u,p) + \int_{\{\gamma(E)\setminus{0}\}} \inf_{p \in C_d} f^2_m(z,u,p,g)\mu(dg) - \sum_{i=1}^d \left( z^i \theta_i + \frac{(\theta_i)^2}{2\lambda} \right)
\end{equation*}
where 

\begin{equation*}
    f^1_m(z, u, p) = \frac{1}{2} \lambda \sum_{i=1}^d \left( p^i \sigma_i 
    - \left( z^i + \frac{\theta_i}{\lambda} \right) \right)^2 \phi_m(z^i) + \int_{\gamma = 0} 
[h_{\lambda}(\phi_m( u(e) - \sum_{i=1}^d p^i\,\eta_i(e) )) - \sum_{i=1}^d p^i\,\eta_i(e)]\nu(de) \notag,
\end{equation*}
and 
\begin{equation*}
    f^2(z, u, p, g) = \int_{\gamma = g} 
[h_{\lambda}(\phi_m( u(e) - \sum_{i=1}^d p^i\,\eta_i(e) )) - \sum_{i=1}^d p^i\,\eta_i(e)]K(g,de) \notag
\end{equation*} for $(z, u, g) \in \R^d \times \R \times \gamma(E) \setminus{\{0\}}$ and $m > 0$.

We proceed exactly as in \cite{turki2026portfolio} and use the equivalence of norms in finite dimension to prove the existence of a solution to this BSDE.\\

In the non-degenerate case, we cannot apply directly the previous truncature structure due to the fact we do not know the sign of the  cross terms of the correlation matrix which  appears. This strategy works well when the matrix $\Sigma$ is spherical in general. 

To handle the general case, we use an inf-convolution approximation only for $f^1$, while $f^2$ is approximated as in the previous case. 

More precisely, for $z \in \mathbb{R}^d$, $u \in (L^2\cap L^\infty)(\nu)$, $p \in C_d$, and $m>0$, we define
\begin{equation*}
f_m^1(z,u,p) = g_m(z,p) 
+ \int_{\gamma = 0} 
[h_{\lambda}(\phi_m( u(e) - \sum_{i=1}^d p^i\,\eta_i(e) )) - \sum_{i=1}^d p^i\,\eta_i(e)]\nu(de)\nu(de),
\end{equation*}
where
\begin{equation*}
g_m(z,p) = \inf_{z' \in \mathbb{R}^d} \left\{ g(z',p) + m\,|z - z'| \right\},
\end{equation*}
and
\begin{equation*}
g(z,p) = f^1(z,u,p) 
- \int_{\gamma=0} 
[h_{\lambda}(\phi_m( u(e) - \sum_{i=1}^d p^i\,\eta_i(e) )) - \sum_{i=1}^d p^i\,\eta_i(e)]\nu(de)\nu(de).
\end{equation*}

We then define the approximate sequence of drivers $(f_m)_m$ by
\begin{align*}
f_m(z,u) 
&= \inf_{p \in C_d} f_m^1(z,u,p) + \int_{\gamma(E)\setminus\{0\}} 
\inf_{p \in C_d} f_m^2(z,u,p,g)\,\mu(dg) - \sum_{i=1}^d \left( z^i \theta_i + \frac{\theta_i^2}{2\lambda} \right).
\end{align*}

The sequence $(f_m)_m$ satisfies the following properties:

\medskip

\noindent
\textbf{1. Lipschitz continuity.} For each $m$, the function $f_m$ is Lipschitz, i.e., there exists a constant $C_m > 0$ such that for all $z,z' \in \mathbb{R}^d$ and $u,u' \in (L^2\cap L^\infty)(\nu)$,
\begin{equation*}
|f_m(z,u) - f_m(z',u')| 
\leq C_m \left( |z - z'| + \|u - u'\|_{L^2(\nu)} \right).
\end{equation*}
This follows from first-order estimates similar to those in lemma ~\eqref{lem-ref-f}, together with the boundedness induced by the truncation and the Cauchy--Schwarz inequality.
\medskip

\noindent
\textbf{2. Monotonicity property.} For each $m$, the function $f_m$ satisfies property (iii) of Lemma~5.2 with the same function $k_m = k$. Therefore, a comparison result holds for the BSDE with driver $f_m$ and terminal condition $F$ (see \cite{Royer06}).

\medskip

\noindent
\textbf{3. Structural property.} For each $m$, the function $f_m$ satisfies point (i) of Lemma~5.2.

\medskip

\noindent
\textbf{4. Convergence.} By construction of the inf-convolution, we have
\[
f_m \uparrow f \quad \mathbb{P}\text{-a.s. as } m \to \infty.
\]
In particular, $(f_m)_m$ is an increasing sequence converging pointwise to $f$.\\

Referring to classical results in the Lipschitz case, we obtain existence and uniqueness of a solution $(Y^m, Z^m, U^m)$ to the BSDE with parameters $(f^m, F)$ (see\cite{delong13} or \cite{oksendal2019applied}). 

Moreover, due to the structural property~3, the solutions are uniformly bounded in their respective spaces $S^\infty \times L^2(W) \times L^2(\widetilde{N})$.

To conclude, it remains to justify the passage to the limit. This follows from the monotone stability result established in \cite{MAM08} (see Lemma~5). 

Finally, by the monotone stability result $(Y^m, Z^m, U^m)$ converges (up to a subsequence) to a triple $(Y,Z,U)$, which is a solution of the BSDE with driver $f$ and terminal condition $F$.\\
For uniqueness, we use a comparison result due to the monotonicity property 2 in of \eqref{lem-ref-f} (see \cite{MAM08, Royer06} for more details).
\end{proof}

\begin{Remark}
Condition (iii) of Lemma~\ref{lem-ref-f} provides a solution of the BSDE~\eqref{BSDEJ} by virtue of Theorem~4.1 in \cite{KTPZ15} only for small terminal conditions. In the general case, the authors impose some regularity properties on the driver, which are not easy to obtain in our context.
\end{Remark}
\section{Power and logarithm utility Maximization} \label{sect4}
\subsection{Investment strategies and wealth dynamic}
In this section, the investment strategy $\alpha_t = (\alpha^1_t, \dots, \alpha^d_t)$ of the investor is expressed in proportion i.e the number of shares of stock $i$ is given by 
$\frac{\alpha^i_t X_{t-}}{S^i_{t-}}$ and constrained to take values in a compact set $\bar{\mathcal{C}} = \bar{\mathcal{C}}_1 \times \dots \times \bar{\mathcal{C}}_d \subset \mathbb{R}^d$.
For an initial endowment $x > 0$, the self-financing wealth process $X^{x,\alpha}$ is defined by $X^{x,\alpha}_0 = x$ and
\begin{equation*}
\frac{d X_{t}^{x,\alpha}}{X^{x,\alpha}_{t-}} 
= \sum_{i=1}^d \alpha^i_t \bigl( \kappa_i\, dt + \sigma_i\, dW^i_t \bigr) 
+ \sum_{i=1}^d  \int_{E} \alpha^i_t \eta_i(e)\, \widetilde{N}(dt,de),
\quad t \in [0,T].
\end{equation*}

To construct the set of admissible strategies from $\mathcal{A}_{\mathrm{sgn}}$, we should impose additional constraints that ensure positivity of the wealth process. Following \cite{bank2022merton}, we introduce bounds when the investor receives a signal $z \neq 0$ on each asset $i$:
\begin{equation*}
\underline{\eta}_i(z) := K(z,.)-\operatorname*{ess\,inf} \eta_i \geq -1,
\qquad
\overline{\eta}_i(z) :=K(z,. )- \operatorname*{ess\,sup} \eta_i,
\end{equation*}
which allows us to construct positive wealth by the interval,
\begin{equation*}
    \Phi^i(z) = ]-\frac{1}{d\overline{\eta}_i(z)}, - \frac{1}{d\underline{\eta}_i(z)}]
\end{equation*} for $z \in \gamma(E) \setminus \{0\}$.\\
To ensure no arbitrage condition, we impose 
\begin{equation}
    \underline{\eta}_i(z) < 0 < \overline{\eta}_i(z).
\end{equation}
In the case the investor receives no-signal, i.e $z=0$ if $\nu(\gamma = 0, \eta_i \neq 0) = 0$ he can take every position of asset $i$ in $\mathbb{R}$ thus $\Phi^i(0) = \R$. By contrast, if $\nu(\gamma = 0, \eta_i \neq 0) > 0$, we ensure positive wealth with the set 
\begin{equation*}
\Phi^i(0) := 
\begin{cases}
\left]-\dfrac{1}{d\bar{\eta}_i(0)},\, -\dfrac{1}{d\underline{\eta}_i(0)}\right], 
& \text{if } \underline{\eta}_i(0) < 0 < \bar{\eta}_i(0), \\[10pt]
\left]-\dfrac{1}{d\bar{\eta}_i(0)},\, +\infty\right), 
& \text{if } 0 \leq \underline{\eta}_i(0) \leq \bar{\eta}_i(0), \\[10pt]
\left(-\infty,\, -\dfrac{1}{d\underline{\eta}_i(0)}\right], 
& \text{if } \underline{\eta}_i(0) \leq \bar{\eta}_i(0) \leq 0,
\end{cases}
\end{equation*}
where the jump bounds are defined as
\begin{align*}
\underline{\eta}_i(0) 
&:= \nu\!\left(\,\cdot \cap \{\gamma = 0,\, \eta_i \neq 0\}\right)\text{-ess\,inf}\; \eta_i 
\;\geq\; -1, \\[4pt]
\bar{\eta}_i(0) 
&:= \nu\!\left(\,\cdot \cap \{\gamma = 0,\, \eta_i \neq 0\}\right)\text{-ess\,sup}\; \eta_i.
\end{align*}
Finally, the set of admissible strategies $\mathcal{A}_{\ell}$ is defined as
\begin{equation*}
    \mathcal{A}_{\ell} := \left\{ \alpha \in {\mathcal{A}}_{\mathrm{sgn}} 
    \;\middle|\; 
    p^i(\Delta G_t) \in \Phi^i(\Delta G_t) \cap \bar{C}_i
    \text{ for all } t \in [0,T] \text{ a.s., } i = 1, \ldots, d
    \right\}.
\end{equation*}

\subsection{Power utility case}
For power utility case, the investor solves the optimization problem 
\beq\label{defV_pow(x)}
    V_{\mathrm{sgn}}(x) & = & \sup_{\alpha \in \mathcal{A}_{\ell}}\E\big[\frac{(X^{x,\alpha}_T)^{\gamma}}{\gamma}\big]\;
\enq 
where the initial capital $x >0$ and $\gamma \in(0, 1)$ represents the risk aversion of the investor.
In this case, Following the approach described in the previous section, we construct the family of processes $(R^\alpha)_{\alpha\in \Ac_\ell }$ by setting
\begin{equation}\label{defRpi}
R^{\alpha}_{t} = \frac{(X^{x,\alpha}_t)^{\gamma}}{\gamma} M_t, \quad t \in [0, T], \; \alpha \in  \Ac_\ell,
\end{equation}
where $dM_t= M_{t-}dY_t$ and $Y_t$ satisfies the BSDEJ 
\begin{align}\label{BSDEpow}
dY_t & =-f(Z_t, U_t) \, dt + \sum_{i =1}^d Z^i_t \, dW^i_t + \int_{E} U_t(e) \, \widetilde{N}(dt, de), \quad t \in [0, T], \\
Y_T &= 0.
\end{align}
 To find  $f$ we proceed as in the the exponential case, let $K^{\alpha}_t = \frac{(X^{x,\alpha}_t)^{\gamma}}{\gamma}$.
Application of It\^o formula to $K^\alpha$ gives 
\begin{align*}
dK^\alpha_t 
&= K^\alpha_{t-} \Bigg[
\Bigg(
\gamma \sum_{i=1}^d p_t^i(0)\,\tilde{\kappa_i} 
+ \frac{\gamma(\gamma - 1)}{2} \sum_{i,j=1}^d \sigma_i \sigma_j \rho_{ij} \, p_t^i(0)\, p_t^j(0)
\Bigg) dt \\
&\quad + \int_E \Big(
\big[1 + \sum_{i=1}^d p_t^i(\gamma(e)) \eta_i(e)\big]^\gamma 
- 1 
\Big)\, \nu(de)\, dt \\
&\quad + \gamma \sum_{i=1}^d \sigma_i p_t^i(0)\, dW_t^i  + \int_E \Big(
\big[1 + \sum_{i=1}^d p_t^i(\gamma(e)) \eta_i(e)\big]^\gamma - 1
\Big)\, \widetilde{N}(dt,de)
\Bigg].
\end{align*}
An other application of Ito rule to the product $K^\alpha_t M_t$ and collecting $dt-$terms from diffusion and jump parts we have 
\begin{align*}
    dR_t = R_{t-}[\Lambda_tdt + \sum_{i=1}^d \Gamma^i_tdW^i_t + \int_E \Theta_t(e)\widetilde{N}(dt, de)]
\end{align*}
with 
\begin{align*}
    \Lambda_t
    &= -f(t, Z_t, U_t)
    + \gamma\sum_{i=1}^d p_t^i(0)\,\tilde{\kappa}_i
    + \frac{\gamma(\gamma-1)}{2}
      \sum_{i,j=1}^d \sigma_i\sigma_j\rho_{ij}\,p_t^i(0)\,p_t^j(0)
    + \gamma\sum_{i,j=1}^d \sigma_i p_t^i(0)\,\rho_{ij}Z_t^j \\
    &\quad + \int_E
      \bigl(1+U_t(e)\bigr)
      \Bigl(\bigl[1+\textstyle\sum_{i=1}^d
        p_t^i(\gamma(e))\eta_i(e)\bigr]^\gamma - 1\Bigr)
      \nu(de), \\[6pt]
    \Gamma_t^i
    &= Z_t^i + \gamma\sigma_i p_t^i(0), \\[6pt]
    \Theta_t(e)
    &= \bigl(1+U_t(e)\bigr)
       \Bigl[1+\sum_{i=1}^d p_t^i(\gamma(e))\eta_i(e)\Bigr]^\gamma - 1.
\end{align*}
for $t\in[0,T]$ and $e\in E  $.
Using Condition iii) of Proposition \ref{MgOptPr} and disintegration formula, we obtain $f$ defined as 
\begin{align}\label{driver_pow}
f(z,u) &= \sup_{p \in \bar{C} \cap \Phi(0)} f^1(z,u,p)
+ \int_{\gamma(E)\setminus\{0\}}
\sup_{p \in \bar{C} \cap \Phi(g)} f^2(g, u, p)\,\mu(dg),
\end{align}
where 
\begin{align}
f^1(z, u, p) &= \gamma \sum_{i=1}^d p^i\,\tilde{\kappa_i} 
+ \frac{\gamma(\gamma - 1)}{2} \sum_{i,j=1}^d \sigma_i \sigma_j \rho_{ij} \, p^i\, p^j  + \gamma \sum_{i,j=1}^d \sigma_i p^i\, \rho_{ij} z^j \\
&\quad + \int_{\gamma=0} \Big(1 + u(e) \Big)
 \Big(
\big[1 + \sum_{i=1}^d p_t^i(\gamma(e)) \eta_i(e)\big]^\gamma - 1
\Big)\, \nu(de),\\
f^2(g, u, p) &= \int_{\gamma=g} 
\Big(1 + u(e) \Big)
 \Big(
\big[1 + \sum_{i=1}^d p_t^i(\gamma(e)) \eta_i(e)\big]^\gamma - 1
\Big)\, K(g, de)
\end{align}
for $z\in\R^d$, $u\in L^2(\nu)$, $g\in\gamma(E)$ and $p\in\R^d$.
Here we note and
\[
\Phi(0) = \Phi^1(0)\times\cdots\times\Phi^d(0),
\qquad
\Phi(g) = \Phi^1(g)\times\cdots\times\Phi^d(g),
\quad g \in \gamma(E)\setminus\{0\}.\]
With this Ansatz, we can see that $f^1$ and $f^2$ are linear in $(z,u)$.
\begin{Proposition}
    The BSDEJ \eqref{BSDEpow} admits a unique solution $(Y,Z,U) \in S^\infty \times L^2(W) \times L^2(\widetilde{N})$.
\end{Proposition}
\begin{proof}
Let $(z, u), (\tilde{z}, \tilde{u}) \in \R^d\times L^2(\nu)$, writing $z = \tilde{z} + z - \tilde{z}, u = \tilde{u} + u - \tilde{u}$ we have 
\begin{align*}
 f^1(z, u, p) &= f^1(\tilde{z}, u, p) + \gamma \sum_{i,j=1}^d \sigma_i p^i\, \rho_{ij} (z^j - \tilde{z}^j) + \int_{\gamma=0}\Big(u(e) - \tilde{u}(e)\Big)\Big(
\big[1 + \sum_{i=1}^d p^i(\gamma(e)) \eta_i(e)\big]^\gamma - 1
\Big)\nu(de), \\
f^2(g, u,p) &= f^2(g, \tilde{u},p) + \int_{\gamma=g}\Big(u(e) - \tilde{u}(e)\Big)\Big(
\big[1 + \sum_{i=1}^d p^i(\gamma(e)) \eta_i(e)\big]^\gamma - 1
\Big)K(g, de).
\end{align*}
Using the compactness of $\bar{C} \cap \Phi(0)$ and $\bar{C} \cap \Phi(g)$, the boundedness of $\eta_i$ and the coefficients $\rho_{ij}$, together with the Cauchy-Schwarz inequality and the subadditivity of the supremum, we can find a constant $C > 0$ such that
\begin{align}
f(z, u) - f(\tilde{z}, \tilde{u}) 
\leq C\bigl( |z - \tilde{z}| + \|u - \tilde{u}\|_{L^2(\nu)} \bigr).
\end{align}
By symmetry, we deduce that $f$ is Lipschitz continuous. Classical results on BSDE with jumps yield existence and uniqueness of solution of \eqref{BSDEpow} in $(Y,Z,U) \in S^2 \times L^2(W) \times L^2(\widetilde{N})$ (see for example theorem 4.5 in \cite{oksendal2019applied}).\\
It remains to prove that $Y \in \mathcal{S}^\infty$. Since $Y \geq 0$, it suffices to derive an upper bound.

First, we observe that
\begin{align}
|f(Z,U)| 
&\leq |f(0,0)| + C\big(|Z| + \|U\|_{L^2(\nu)}\big) \\
&\leq |f(0,0)| + C + \frac{C}{2}\big(|Z|^2 + \|U\|_{L^2(\nu)}^2\big).
\end{align}

Taking conditional expectation in \eqref{BSDEpow} with respect to $\mathcal{F}_t$ yields
\begin{equation} \label{first_step}
Y_t \leq 1 + (T-t)\big(C + |f(0,0)|\big) 
+ \frac{C}{2}\,\mathbb{E}_t\!\left[
\int_t^T |Z_s|^2 ds 
+ \int_t^T \int_E U_s^2(e)\,\nu(de)\,ds
\right].
\end{equation}

Let $\beta > 0$ (to be fixed later). The main idea is to obtain an upper bound for
\[
\mathbb{E}_t\!\left[
\int_t^T e^{\beta s}|Z_s|^2 ds 
+ \int_t^T \int_E e^{\beta s} U_s^2(e)\,\nu(de)\,ds
\right]
\]
in terms of $Y_t$. 

On the one hand, applying Itô's formula to $e^{\beta t} Y_t^2$ and taking conditional expectation yields
\begin{align}\label{a_priori_est}
e^{\beta T} 
+ 2\,\mathbb{E}_t\!\left[\int_t^T e^{\beta s} Y_s f(Z_s,U_s)\,ds\right]
&= \mathbb{E}_t\!\left[\int_t^T e^{\beta s} \Big( 
\beta Y_s^2 + Z_s^\top \Sigma Z_s + \int_E U_s^2(e)\,\nu(de)
\Big)\,ds \right] \\
&\leq \mathbb{E}_t\!\left[\int_t^T e^{\beta s} \Big(
\beta Y_s^2 + \lambda_{\min} |Z_s|^2 + \int_E U_s^2(e)\,\nu(de)
\Big)\,ds \right],
\end{align}
where $\lambda_{\min}$ denotes the smallest eigenvalue of the correlation matrix $\Sigma$.

Using Young's inequality, for all $a,b \in \mathbb{R}$ and $\varepsilon > 0$,
\[
2ab \leq \frac{a^2}{\varepsilon} + \varepsilon b^2,
\]
we obtain
\begin{align}
2 Y_s f(Z_s,U_s)
&\leq 2|f(0,0)| Y_s + 2C Y_s \big(|Z_s| + \|U_s\|_{L^2(\nu)}\big) \\
&\leq 2|f(0,0)| Y_s 
+ \frac{2C^2}{\varepsilon} Y_s^2 
+ \varepsilon \big(|Z_s|^2 + \|U_s\|_{L^2(\nu)}^2\big).
\end{align}

Therefore,
\begin{align}
\mathbb{E}_t\!\Bigg[\int_t^T e^{\beta s} \Big(
(\beta - \tfrac{C^2}{\varepsilon}) Y_s^2 
+ (\lambda_{\min} - \varepsilon) |Z_s|^2 
+ (1 - \varepsilon) \int_E U_s^2(e)\,\nu(de)
\Big)\,ds \Bigg]
\leq e^{\beta T} + 2|f(0,0)| \mathbb{E}_t\!\left[\int_t^T Y_s\,ds\right].
\end{align}

Choosing $\varepsilon < \min(\lambda_{\min},1)$ and $\beta > \frac{2C^2}{\varepsilon}$, we deduce the existence of a constant $C_{\beta,\varepsilon} > 0$ such that
\begin{align}
\mathbb{E}_t\!\Bigg[
\int_t^T e^{\beta s} \big(Y_s^2 + |Z_s|^2\big)\,ds
+ \int_t^T \int_E e^{\beta s} U_s^2(e)\,\nu(de)\,ds
\Bigg]
\leq C_{\beta,\varepsilon}
\left(
e^{\beta T} + 2|f(0,0)| \mathbb{E}_t\!\left[\int_t^T Y_s\,ds\right]
\right).
\end{align}

Combining this estimate with \eqref{first_step} and applying Grönwall's lemma, we conclude that $Y \in \mathcal{S}^\infty$.

\end{proof}

\begin{Remark}
    Furthermore, from \eqref{a_priori_est}, there exists a constant $C_1 > 0$ such that for all stopping times $\tau$,
\begin{align}
\mathbb{E}_{\tau}\!\left[
\int_{\tau}^T e^{\beta s} \Big(
\beta Y_s^2 + \lambda_{\min} |Z_s|^2 + \int_E U_s^2(e)\,\nu(de)
\Big)\,ds
\right] \leq C_1.
\end{align}

In particular, the martingales
\[
\sum_{i=1}^d \int_0^. Z_s^i \, dW_s^i
\quad \text{and} \quad
\int_0^. \int_E U_s(e)\,\widetilde{N}(ds,de)
\]
are BMO martingales.
\end{Remark}
We characterize optimal strategies by the following theorem 
\begin{Theorem}
Let $(Y,Z,U)\in S^\infty \times L^2(W) \times L^2(\widetilde{N})$ the solution of \eqref{BSDEpow}.Then, there exists a Borel map 
$p^*:~\R\times\R^d\times  L^2(\nu)\rightarrow\R^d$  satisfying
\begin{align} \label{mindriveratpistar_pow}
f(z, u) & =  f^1(z,u,p^*(0,z,u))+
\int_{\gamma(E) \setminus \{0\}}  f^2(g,u,p^*(g,z,u))
\mu(dg)
\end{align}
and
\begin{align}
    p^*(g,z,u) & \in   \bar{C} \cap \Phi(g)
\end{align}
    
for all $g,z,u\in\gamma(E)\times\R ^d\times L^2(\nu)$, and we have $V(x)=\frac{x^\gamma}{\gamma}\exp(Y_0)$ and $\alpha^*=(p^*(\Delta G_t,Z_t,U_t))_{t\in[0,T]}$ is an optimal strategy in $\Ac_l$.
\end{Theorem}
\begin{proof}
Let $(Y,Z,U)\in S^\infty \times L^2(W) \times L^2(\widetilde{N})$  be the solution  BSDEJ \eqref{BSDEpow} with $f$ given by \eqref{driver_pow} and consider the family $(R^\alpha)_\alpha$ defined by \eqref{defRpi}.

The family $(R^\alpha)_\alpha$  satisfies (i) and (ii) of Proposition \ref{MgOptPr} as for any $\alpha\in \Ac_\ell$.  From the definition of $f$ and the computations made above, we observe that $R^\alpha$  is a nonnegative local supermartingale. Hence by Fatou's Lemma it is a supermatingale and (iii) of  Proposition \ref{MgOptPr} is satisfied. Suppose there exists a Borel function $p^*$ satisfying 
\eqref{mindriveratpistar_pow}. For $\alpha = \alpha^*$, the process $(R^{\alpha^*}_t)_{t \in [0,T]}$ is a local martingale. 
We recall that 
$\sum_{i=1}^d \int_0^. Z_s^i \, dW_s^i$ and 
$\int_0^. \int_E U_s(e)\,\widetilde{N}(ds,de)$ are BMO martingales.

Since $p$ takes values in a compact set, it follows that 
$\sum_{i=1}^d \int_0^. \Gamma_t^i \, dW_t^i$ is also a BMO martingale. 
Moreover, by Young's inequality and using the fact that 
$p_t^i(\gamma(e))$ and $\eta_i(e)$ take values in a compact set for all $e \in E$ and $t \in [0,T]$, 
we deduce that the martingale 
$\int_0^. \int_E \Theta_t(e)\,\widetilde{N}(dt,de)$ is of BMO type.

Finally, applying Kazamaki's criterion \eqref{Kazamaki_lemma}, 
we conclude that $(R^{\alpha^*}_t)_{t \in [0,T]}$ is a true martingale. It remains to show existence and measurability property of $p^*$. Let $(z, u) \in \mathbb{R}^d \times L^2(\nu)$ by the dominated convergence theorem, the maps $p \mapsto f^1(z, u, p)$ and $p \mapsto f^2(g, z, u, p)$ are continuous. By compactness of the optimization sets $\bar{C} \cap \Phi(0)$ and $\bar{C} \cap \Phi(g)$ their respective maximizers $p^*(0, z, u)$ and $p^*(g, z, u)$ are attained. Since $\gamma(E)$ is a Borel subset of $\mathbb{R}$ by assumption, $f^1$ and $f^2$ are Carathéodory functions i.e continuous in $p$ and measurable in the other variables. Therefore
the maps $g \mapsto p^*(g,z, u)$ is a Borel measurable due to theorem 18.19 in \cite{aliprantis2006infinite}.
\end{proof}
\subsection{Logarithm utility case} 

For the logarithm utility function we consider the set of admissible strategies as in the power utility case and assume a deterministic time dependent volatility. The optimization problem is given by 
\beq\label{defV_l(x)}
    V_{\mathrm{sgn}}(x) & = & \sup_{\alpha \in \mathcal{A}_{\ell}}\E\big[\log(X^{x,\alpha}_T)\big]\;
\enq 
where the initial capital $x >0$. We apply the martingale optimality principle as in the previous section which consists in find a process $R^{\alpha}$ with $R^{\alpha}_T = \log(X^{\alpha}_T)$ and an initial value that is independent of $\alpha$. Moreover, $R^{\alpha}$ is a supermartingale for all $\alpha \in \mathcal{A}_{\ell}$ and there exists a $\alpha^*$ such that $R^{\alpha^*}$ is a martingale. Then $\alpha^*$ is an optimal strategy and $R_0^{\alpha^*}$ is the value function of the optimization problem \eqref{defV_l(x)}.
For $x >-1$, we introduce 
\begin{equation*}
k(x) = \log(1+x).  
\end{equation*}
An application of Itô's formula gives
\begin{align*}
d\log X^{x,\alpha}_t &= \left[
\langle p_t(0),\, \tilde{\kappa} \rangle 
- \frac{1}{2} p_t(0)^\top \operatorname{diag}(\sigma_t) \Sigma \operatorname{diag}(\sigma_t)\,p_t(0) 
+ \int_E k\bigl(\langle p_t(0),\, \eta(e) \rangle\bigr)\,\nu(de)
\right]dt \\
&\quad + \bigl\langle p_t(0)\operatorname{diag}(\sigma_t),\, dW_t \bigr\rangle 
+ \int_E k\!\left(\langle p_t(\gamma(e)),\, \eta(e) \rangle\right)\tilde{N}(dt,\, de)
\end{align*} for $t \in [0, T]$.\\

We define the process $R^\alpha$ by, for all $t \in [0,T]$,
\begin{equation*}
R^\alpha_t = \log\bigl(X^{x,\alpha}_t\bigr) + Y_t,
\end{equation*}
where
\begin{equation*}
Y_t = \int_t^T f(s)\,ds 
+ \int_t^T \langle Z_s,\, dW_s \rangle 
+ \int_t^T \int_{E} U_s(e)\,\tilde{N}(ds,de).
\end{equation*}

Using the previous representation and the disintegration formula, we obtain
\begin{align}\label{driver_log}
f(t) &= \sup_{p \in \bar{C} \cap \Phi(0)} f^1(t, p)
+ \int_{\gamma(E)\setminus\{0\}}
\sup_{p \in \bar{C} \cap \Phi(g)} f^2(p,g)\,\mu(dg),
\end{align}
where
\begin{align*}
f^1(t,p) &:= \langle p,\,\tilde{\kappa}\rangle
- \frac{1}{2} p^\top \operatorname{diag}(\sigma_t) \Sigma \operatorname{diag}(\sigma_t)\,p 
+ \int_{\{\gamma=0\}} k\bigl(\langle p,\,\eta(e)\rangle\bigr)\,\nu(de),\\[4pt]
f^2(p,g) &:= \int_{\{\gamma=g\}}
k\bigl(\langle p,\,\eta(e)\rangle\bigr)\,K(g,de),
\end{align*}
and
\[
\Phi(0) = \Phi^1(0)\times\cdots\times\Phi^d(0),
\qquad
\Phi(g) = \Phi^1(g)\times\cdots\times\Phi^d(g),
\quad g \in \gamma(E)\setminus\{0\}.
\]

In particular, the driver $f$ does not depend on $(Y,Z,U)$.

\begin{Theorem}\label{thm:log}
Suppose the BSDEJ with terminal condition $\xi = F$ and driver $f$ given by
\eqref{driver_log} admits a solution
$(Y,Z,U) \in \mathcal{S}^2 \times L^2(W) \times L^2(\widetilde{N})$.
Then there exists a Borel map $p^* : \gamma(E) \to \mathbb{R}^d$ satisfying
\begin{equation}\label{mindriveratpistar}
f(t) = f^1\!\left(p^*_t(0)\right)
+ \int_{\gamma(E)\setminus\{0\}} f^2\!\left(p^*_t(g),\,g\right)\mu(dg),
\end{equation}
with $p^*_t(g) \in \bar{C} \cap \Phi(g)$ for all $g \in \gamma(E)$.
Moreover, $V(x) = \log(x) + Y_0$ and
$\alpha^* = (p^*_t(\Delta G_t))_{t\in[0,T]}$ is optimal in $\mathcal{A}_\ell$.
\end{Theorem}

\begin{proof}
For any $\alpha \in \mathcal{A}_\ell$, the process $R^\alpha$ is a supermartingale
from the above computations combined with the boundedness of $p \in
\bar{C} \cap \Phi(\Delta G_t)$. At $\alpha = \alpha^*$, $R^{\alpha^*}$ is a
true martingale by construction.

It remains to show existence and measurability of $p^*$. By the dominated
convergence theorem, the maps $p \mapsto f^1(p)$ and $p \mapsto f^2(p,g)$
are twice continuously differentiable, with Hessians
\begin{align*}
D^2_p f^1(p) &= -\,\mathrm{diag}(\sigma)
+ \int_{\{\gamma=0\}} k''\!\bigl(\langle p,\,\eta(e)\rangle\bigr)
\,\mathrm{diag}\!\left(\eta_1(e)^2,\ldots,\eta_d(e)^2\right)\nu(de),\\[4pt]
D^2_p f^2(p,g) &= \int_{\{\gamma=g\}} k''\!\bigl(\langle p,\,\eta(e)\rangle\bigr)
\,\mathrm{diag}\!\left(\eta_1(e)^2,\ldots,\eta_d(e)^2\right)K(g,de),
\end{align*}
where $k''(x) = -1/(1+x)^2 < 0$ for $x > -1$. Since both Hessians are
negative definite (under assumption \eqref{assump}), the maps $p \mapsto f^1(p)$ and
$p \mapsto f^2(p,g)$ are strictly concave in $p$. By compactness of
$\bar{C} \cap \Phi(0), \bar{C} \cap \Phi(g)$ the suprema in \eqref{mindriveratpistar} are
attained at unique maximizers $p^*(0)$ and $p^*(g)$ for $f^1$ and $f^2$
respectively. Since $\gamma(E)$ is a Borel subset of $\mathbb{R}$ by assumption, $f^1$ and $f^2$ are Carathéodory functions i.e continuous in $p$ and measurable in the other variables. Therefore
the maps $g \mapsto p^*(g,z, u)$ is a Borel measurable due to theorem 18.19 in \cite{aliprantis2006infinite}.
\end{proof}
In fact we have existence and uniqueness of the previous BSDEJ as the driver does not depend on $(Y, Z, U)$ and is measurable cf \cite{delong13}.

\section{Individuals jumps signal}\label{sect6}
In this section, we study the idiosyncratic case and consider the Merton 
problem when the investor receives different signal from each asset 
$i \in \{1, \ldots, d\}$ in the exponential case. 
\subsection{Market model}

We consider the same probability space $(\Omega, \mathcal{A}, \mathbb{P})$ 
and the $d$-dimensional correlated Brownian motion 
$(W^1_t, \ldots, W^d_t)_{t \geq 0}$ as in Section~\ref{sect2}. In addition, 
we now consider $d$ \emph{independent} homogeneous Poisson random measures 
$N^1, \ldots, N^d$, where each $N^i$ is defined on $\mathbb{R}_+ \times E$ 
with compensator $dt \otimes \nu_i(de)$, and $\nu_i$ is a finite positive 
measure on $(E, \mathcal{E})$ satisfying
\begin{equation*}
    \nu_i(\{0\}) = 0, \quad i = 1, \ldots, d.
\end{equation*}
We denote by $\widetilde{N}^i(dt, de) := N^i(dt, de) - \nu_i(de)\,dt$ the 
compensated measure of $N^i$. The filtration $\mathbb{F} = (\mathcal{F}_t)_{t 
\geq 0}$ is the $\mathbb{P}$-completion of the natural filtration generated by 
$W^1, \ldots, W^d$ and $N^1, \ldots, N^d$ and satisfies the usual conditions.

\begin{Remark}
The key difference with Section~\ref{sect2} is the use of $d$ 
\emph{independent} Poisson measures $N^1, \ldots, N^d$ instead of a single 
common measure $N$. Each asset $S^i$ is now subject to its own idiosyncratic 
jump risk, independent of the other assets. This models situations where jumps 
are driven by asset-specific local events (earnings surprises, single-name credit 
events for example) rather than market-wide shocks.
\end{Remark}

The financial market consists of a riskless asset with zero interest rate and 
$d$ risky assets $S^1, \ldots, S^d$ whose dynamics are given by
\begin{equation}\label{eq:SDE_idio}
    \frac{dS^i_t}{S^i_{t-}} = \kappa_i\,dt + \sigma_i\,dW^i_t 
    + \int_E \eta_i(e)\,\widetilde{N}^i(dt, de), 
    \quad t \in [0, T],\quad i = 1, \ldots, d,
\end{equation}
with $S^i_0 = s_i > 0$, where $\kappa_i \in \mathbb{R}$  
$\sigma_i > 0$ , $\theta_i := \kappa_i / \sigma_i$ , and $\eta_i : E \to \mathbb{R}$ satisfying the same assumptions as in \eqref{sect2}, for 
$i = 1, \ldots, d$.
\subsection{Portfolio with signal strategies and exponential utility problem} We consider investment strategies where the investor has access to different extra information on all the stock $S^i$.  More precisely, we suppose that
the extra information on the stock $S^i$ is given by an impending signal process $G^i$ defined by
\begin{equation*}
G^i_{t}   =  \int_{0}^{t} \int_{E}\gamma_i(e) \widetilde{N}^i(ds, de)\;,\quad t \geq 0\;, 
\end{equation*}
where
 $\gamma_i:E \xrightarrow{}{\mathbb{R}}$ is a Borel map such that there exists a constant $C$ satisfying 
 \begin{equation}
\int_E |\gamma_i(e)|^2 \nu_i(de)  <  \infty
 \end{equation} for $i = 1 , \cdots d.$
We also assume that $\gamma^i(E)$ is a Borel subset of $\mathbb{R}$.
We would like to consider strategies $\pi = (\pi^1, \ldots, \pi^d)$ such that for each $i$
\begin{equation}
\pi^i=(\pi^i_t)_{t\in[0,T]} \text{ is }\mathcal{P} \vee \sigma(G^i)\text{-measurable and satisfy  } \eqref{cond-Adm-pi}\;.
\end{equation}
 
By \cite[Lemma 2.2]{bank2022merton} a process $\pi^i$ is $\mathcal{P} \vee \sigma(G^i)$-measurable if and only if there exists maps $p^i:~\R_+\times\Omega\times\R\rightarrow\R$ which is 
 $\mathcal{P} \otimes \mathcal{B}(\mathbb R)$-measurable, such that
\begin{equation}
\pi^i_t  =  p^i_t(\Delta_t G^i)\;,\quad t\geq0 ~\P-a.s.
\end{equation} for $i = 1, ..., d$
where $\Delta_t G^i=G^i_t-G^i_{t-}$ for $t\geq0$.
In particular, we can rewrite the set of such strategies as
\begin{equation}
 \text{Processes }(\pi_t=(p^1_t(\Delta G^1_t), ..., p^d_t(\Delta G^d_t))_{t\in[0,T]} \text{ satisfying  } \eqref{cond-Adm-pi} \mbox{ with } p^i  \text{ being }\mathcal{P} \otimes \mathcal{B}(\mathbb{R})\text{-measurable} \;.
\end{equation}

We define for each $i$ the kernel $K^i:~\mathbb{R}\times \Bc(E)\rightarrow \mathbb{R} $ such that 
\begin{equation}
\nu_i\big(de \cap \{\gamma_i \neq 0\}\big) = \int_{\gamma_i(E) \setminus\{0\}} K^i(g,de)\mu^i(dg),  
\end{equation}
 where $\mu^i$ is the measure image of $\nu_i $ by $\gamma$, $i.e.$ $\mu = \nu \circ \gamma_i^{-1}$. To get such a kernel $K^i$ see \eqref{sect3}.
 
Finally, we shall consider the set of admissible strategies with signal $\mathcal{A}_{\mathrm{sgn}}$ defined by
\begin{equation*}
\Ac_{\mathrm{sgn}}  = \Big\{\text{Processes }(\pi_t=(p^1_t(\Delta G^1_t), ..., p^d_t(\Delta G^d_t)))_{t\in[0,T]} \text{ satisfying  } \eqref{cond-Adm-pi} \mbox{ with } p^1, ..., p^d  \text{ being }\mathcal{P} \otimes \mathcal{B}(\mathbb{R})\text{-measurable}\Big\} \;.
\end{equation*}
For $\pi\in \mathcal{A}_{\mathrm{sgn}}$ and for an initial endowment $x\in\mathbb{R}$, the self financing wealth process $ X^{x,\pi}$ defined by $ X_0^{x,\pi}=x$ and
\begin{equation}\label{defDynX_idio}
           d X_{t}^{x, \pi}  
           = \sum_{i=1}^d \pi^i_t[\kappa_i dt + \sigma_i dW^i_t]  +  \int_{E}\pi^i_t\eta^i(e) \widetilde{N}^i(dt, de),\quad  t\in [0, T].
\end{equation}

We keep the same notations as in \eqref{sect4} and with the new features we define \begin{align*}
\bar{\mathcal{A}}_{\mathrm{sgn}} &= \Bigl\{ 
    \pi = \bigl(p^1_t(\Delta G^1_t), \ldots, p^d_t(\Delta G^d_t)\bigr)_{t\in[0,T]} 
    \;\Big|\; 
    p^1, \ldots, p^d \text{ are } \mathcal{P}\otimes\mathcal{B}(\mathbb{R})\text{-measurable}
    \text{ and satisfy } \eqref{limited-credit-line}
\Bigr\}.
\end{align*}
We recall the value function $V_{\mathrm{sgn}}$  by
\begin{equation}\label{defV(x)}
    V_{\mathrm{sgn}}(x)  =  \sup_{\pi \in \bar\Ac_{\mathrm{sgn}}}\E\big[U_{\lambda}(X^{x,\pi}_T - F)\big]\;,
\end{equation}
for $x\in\mathbb{R}$ where $U_{\lambda}$ is the exponential utility function with risk aversion $\lambda$.
For $i = 1, \dots, d$ we introduce the following space
\begin{itemize}
    
    \item \( L^2(\widetilde{N}^i) = \left\{\mathcal{P} \otimes \mathcal{B}(E)\text{-measurable processes } U \text{ valued in $\R$ such that } \mathbb{E} \left[\int_0^T \int_E U_t^2(e) \, \nu_i(de) \, dt \right] < \infty \right\} \).
    
    \item $L^2(\nu_i) = \{ u : E \xrightarrow{} \mathbb{R} \mbox{ Borel, } \int_{E}u^2(e)\nu_i(de) <\infty \}$. 

\end{itemize}
We use \eqref{MgOptPr} as in \eqref{sect3} by setting
\begin{equation}\label{defRpi}
R^{\pi}_{t} = -\exp\left(-\lambda\left(X^{x,\pi}_{t} - Y_t\right)\right),
\quad t \in [0, T], \quad \pi \in \bar{\mathcal{A}}_{\mathrm{sgn}},
\end{equation}
where $(Y, Z, (U^i)_{i=1}^d) \in S^\infty \times L^2(W) \times \prod_{i=1}^d L^2(\widetilde{N}^i)$ is a solution to the BSDE with jumps
\begin{equation}\label{BSDEexp_idio}
\begin{aligned}
Y_t
&= F + \int_t^T f(s, Z_s, U^1_s, \ldots, U^d_s)\, ds
- \sum_{i=1}^d \int_t^T Z^i_s \, dW^i_s \\
&\quad - \sum_{i=1}^d \int_t^T \int_E U^i_s(e)\, \widetilde{N}^i(ds,de),
\qquad t \in [0, T].
\end{aligned}
\end{equation}

Following the same steps as in Section~\eqref{sect3} (i.e., applying It\^o's formula and using the disintegration property of each measure $\nu_i$), we define the driver $f$ by
\begin{align}
f(z, (u^i)_{i=1}^d)
&= \inf_{p \in C_d} f^1(z,(u^i)_{i=1}^d,p)
+ \sum_{i=1}^d \int_{\gamma_i(E)\setminus\{0\}} \inf_{p \in C_d} f^2(g_i,u^i,p)\, \mu^i(dg) \notag \\
&\quad - \sum_{i=1}^d \left( z^i \theta_i + \frac{\theta_i^2}{2\lambda} \right).
\label{defDriverfexp_idio}
\end{align}

The function $f^1$ is given by
\begin{align}
f^1(z,(u^i)_{i=1}^d,p)
&= \frac{\lambda}{2} \Bigg[
\sum_{i=1}^d \left( p^i\sigma_i - \left( z^i + \frac{\theta_i}{\lambda} \right) \right)^2 \notag  + \sum_{\substack{i,j=1 \\ i \neq j}}^d 
\left( p^i\sigma_i - z^i \right)
\left( p^j\sigma_j - z^j \right)\rho_{ij}
\Bigg] \notag \\
&\quad + \sum_{i=1}^d \int_{\gamma_i = 0}
\Big(h_{\lambda}( u^i(e) - \sum_{j=1}^d p^j \eta_j(e) ) - \sum_{j=1}^d p^j \eta_j(e) \Big)\nu_i(de).
\end{align}

Moreover,
\begin{equation}
f^2(g_i,u^i,p)
= \int_{\gamma_i = g_i}
\Big(h_{\lambda}( u^i(e) - \sum_{j=1}^d p^j \eta_j(e) ) - \sum_{j=1}^d p^j \eta_j(e) \Big)\nu_i(de) K^i(g_i,de),
\end{equation}
for $z \in \mathbb{R}^d$, $u^i \in (L^2 \cap L^\infty)(\nu_i)$, $g_i \in \gamma_i(E)$, and $p \in \mathbb{R}^d$, $i = 1, \ldots, d$.

\subsection{Characterization of the optimal strategies}

\begin{Theorem}
Suppose that the BSDE \eqref{BSDEexp_idio} with terminal condition $\xi = F$ and driver $f$ given by \eqref{defDriverfexp_idio} admits a solution
\[
(Y, Z, (U^i)_{i=1}^d) \in S^\infty \times L^2(W) \times \prod_{i=1}^d L^2(\widetilde{N}^i).
\]
Then there exists a Borel map
\[
p^* : \mathbb{R}^d \times \mathbb{R}^d \times \prod_{i=1}^d L^2(\nu_i) \to \mathbb{R}^d
\]
such that
\begin{equation*}
p^*((g_i)_{i=1}^d, z, (u^i)_{i=1}^d) \in C_d,
\end{equation*}
for all $(g_i)_{i=1}^d \in \prod_{i=1}^d \gamma_i(E)$, $z \in \mathbb{R}^d$, and $u^i \in (L^2 \cap L^\infty)(\nu_i)$. Moreover,
\[
V(x) = U_{\lambda}(x - Y_0),
\]
and the strategy
\[
\pi^*_t = p^*(\Delta G^1_t, \ldots, \Delta G^d_t, Z_t, (U^i_t)_{i=1}^d),
\quad t \in [0,T],
\]
is optimal in $\bar{\mathcal{A}}_{\mathrm{sgn}}$.
\end{Theorem}
\begin{proof}
We follow the same steps as in \eqref{iden_signal_case} and we are done with additivity.    
\end{proof}
\subsection{Existence and uniqueness to the BSDEJ}
A strategy to prove the existence and uniqueness of a solution to the BSDEJ \eqref{BSDEexp_idio} consists in applying a transformation on the mark space and the Poisson measure  that reduces it to a classical BSDEJ, and then using the same truncation procedure as in \eqref{existence_solution} to conclude.\\
Let us introduce the product mark space $\widetilde{E} := \{1, \ldots, d\} \times E$
and define the random measure $N$ on 
$\mathbb{R}_+ \times \widetilde{E}$ by
\begin{align}
    N\bigl(ds,\, d(k,e)\bigr) &:= N^k(ds,\, de),
    \label{eq:aggregated_N}
\end{align}
with compensator
\begin{align}
    \nu\bigl(d(k,e)\bigr) 
    &:= \sum_{i=1}^d \delta_i(dk)\,\nu_i(de)
     = \delta_k(dk)\,\nu_k(de),
    \label{eq:aggregated_nu}
\end{align}
so that the compensated measure is
\begin{align}
    \widetilde{N}\bigl(ds,\, d(k,e)\bigr) 
    &:= N\bigl(ds,\, d(k,e)\bigr) - \nu\bigl(d(k,e)\bigr)\,ds
    \notag\\
    &\phantom{:}= N^k(ds,\, de) - \nu_k(de)\,ds
     = \widetilde{N}^k(ds,\, de).
    \label{eq:aggregated_Nt}
\end{align}
Due to the independence of each measure $N^i$, this measure defined on $\mathbb{R}_+ \times\widetilde{E}$ is a Poisson random measure with the compensator defined by \eqref{eq:aggregated_nu}.
We can rewrite the BSDEJ \eqref{BSDEexp_idio} as follow 
\begin{align*}
Y_t
&= F + \int_t^T \tilde{f}(s, Z_s, U_s)\, ds
- \sum_{i=1}^d \int_t^T Z^i_s \, dW^i_s \\
&\quad - \int_t^T \int_{\widetilde{E}} U_s(i,e)\, \widetilde{N}^i(ds,d(i,e)),
\qquad t \in [0, T].
\end{align*}
with $\tilde{f}(t, Z_t, U_t) = f(t, Z_t, U^1_t, \cdots, U^d_t)$.
With these notations, we use the norm of $U$ given by $$||U||^2_{L^2(\nu)} = \sum_{i=1}^d ||U^i||^2_{L^2(\nu_i)}$$ and obtain we have the desired result.

\section{Numerical Illustration}\label{sect7}
In this section, we consider the case of $d = 2$ assets with correlation matrix
\[
\Sigma =
\begin{pmatrix}
1 & \rho \\
\rho & 1
\end{pmatrix},
\qquad \rho \in [-1,1].
 \]
We take $E = \mathbb{R}^2$ equipped with the $\sigma$-field $\mathcal{E} = \mathcal{B}(\mathbb{R}^2)$. The marked random measure $N(dt,de)$ has a compensator the finite L\'evy measure
\[
\nu = \lambda\, \mathcal{N}(0,1) \otimes \mathcal{N}(0,1).
\]
For $e = (e_1,e_2) \in E$, we define
\[
\eta_1(e) = \tanh(a_1 e_1), \qquad
\eta_2(e) = \tanh(a_2 e_1),
\]
and the signal
\[
\gamma(e) = \alpha e_1 + e_2,
\]
where $(a_1,a_2) \in \mathbb{R}^2$ are constants and $\alpha \in \mathbb{R}$ represents the strength of the signal.

In particular, for all $e \in E$ and $i=1,2$, we have
\[
|\eta_i(e)| \leq 1.
\] Note that $\nu(\{\eta_i <-1\}) = 0$.
The measure $\nu$ is disintegrated as:
\[
\mu(dg) = \lambda \,\mathcal{N}_{0,\, 1+\alpha^2}(dg)
\]
\[
K(g, de_1\, de_2) = \mathcal{N}_{\frac{\alpha g}{1+\alpha^2},\, \frac{1}{1+\alpha^2}}(de_1) 
\otimes \delta_{g - \alpha e_1}(de_2)
\]
For $g \in  \R \setminus\{0\} $, the constraint set is:
\[
\Phi(g) = \left]-\tfrac{1}{2}, \tfrac{1}{2}\right] \times \left]-\tfrac{1}{2}, \tfrac{1}{2}\right]
\] and 
\[
\Phi(0)= \R
.\] Moreover we fix the constraint set $\bar{C} = [0,1] \times [0,1]$.

\medskip
\noindent\textbf{Optimal strategy without signal.}
The optimal strategy $p^*(0)$ is given by:
\[
p^*(0) = \argmax_{p \in \bar{C}} \left\{ \langle p,\, \tilde{\kappa} \rangle 
- \frac{1}{2}p^\top \operatorname{diag}(\sigma) \Sigma \operatorname{diag}(\sigma)\,p \right\}
\]
Since $\bar{C}$ is a box constraint, the unconstrained optimizer 
$\operatorname{diag}(\sigma)^{-1}\Sigma^{-1}\operatorname{diag}(\sigma)^{-1}\tilde{\kappa}$ 
is projected onto $\bar{C}$, giving the explicit solution:
\[
p^*(0) = \Pi_{\bar{C}}\!\left(\operatorname{diag}(\sigma)^{-1}\,\Sigma^{-1}\,
\operatorname{diag}(\sigma)^{-1}\tilde{\kappa}\right)
\]

\medskip
\noindent\textbf{Optimal strategy with signal.}
When the investor receives a signal $g \neq 0$, he solves:

\label{eq:optim_psignal}
\begin{equation*}
    p^*(g) = \argmax_{p \in \bar{C} \cap\Phi(g)} \int_{\mathbb{R}} 
\log\!\left(1 + p_1\,\eta_1(e_1) + p_2\,\eta_2(e_1)\right) 
\mathcal{N}_{\frac{\alpha g}{1+\alpha^2},\, \frac{1}{1+\alpha^2}}(de_1)
\end{equation*}

Problem~\eqref{eq:optim_psignal} does not admit a closed-form solution, 
and we therefore rely on numerical methods to study the impact of the 
signal on the optimal strategies.

\medskip
\noindent\textbf{Numerical experiments.} In the following we fix values of the parameters $$
\kappa_1 = 0.07, \kappa_2 =  0.2, \sigma_1 = 0.2,\sigma_2 = 0.4, \lambda=5,  s_1 = s_2=1, a_1 = 0.3, a_2=0.4 , \rho = 0.5. \mbox{ and } T = 1.$$
Figure~2 illustrates the evolution of the optimal fractions $(p_1(g), p_2(g))$ as functions of the received signal $g$ and the signal strength $\alpha$. As expected, for positive values of $\alpha$, the optimal portfolio is non-decreasing in $g$. When the signal is sufficiently reliable, the investor allocates wealth equally between the two assets. In contrast, when the signal is less reliable, the investor reduces exposure and may fully disinvest, i.e., $(p_1(g), p_2(g)) = (0, 0)$ (see \cite{bank2022merton} for more details).

Figure~3 shows that the expected utility is non-increasing with respect to the correlation parameter $\rho$. A higher positive correlation reduces diversification benefits, as asset returns become more aligned, effectively concentrating risk. Conversely, negative correlation improves diversification by offsetting fluctuations between assets, thereby reducing portfolio variance and increasing expected utility.
\begin{figure}[h!]
    \centering
    \includegraphics[width=0.6\textwidth]{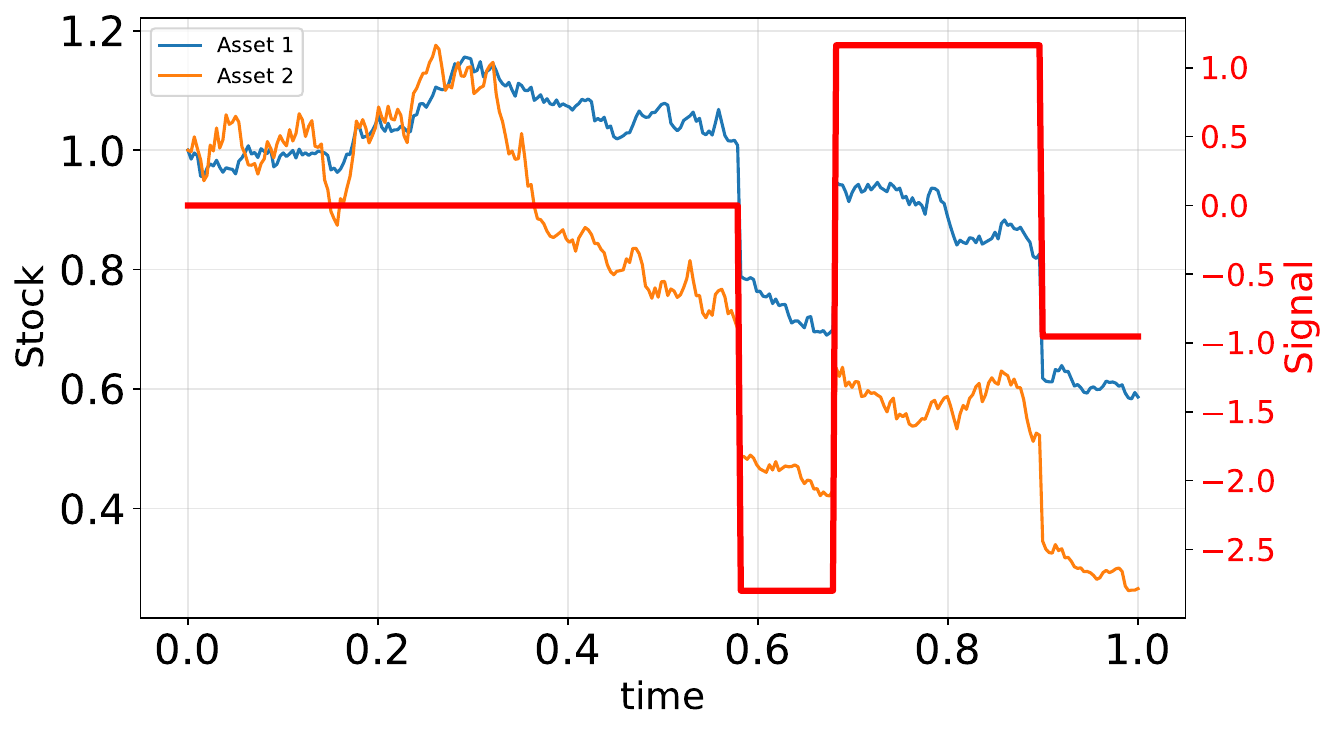}

    \centering
    \includegraphics[width=0.6\textwidth]{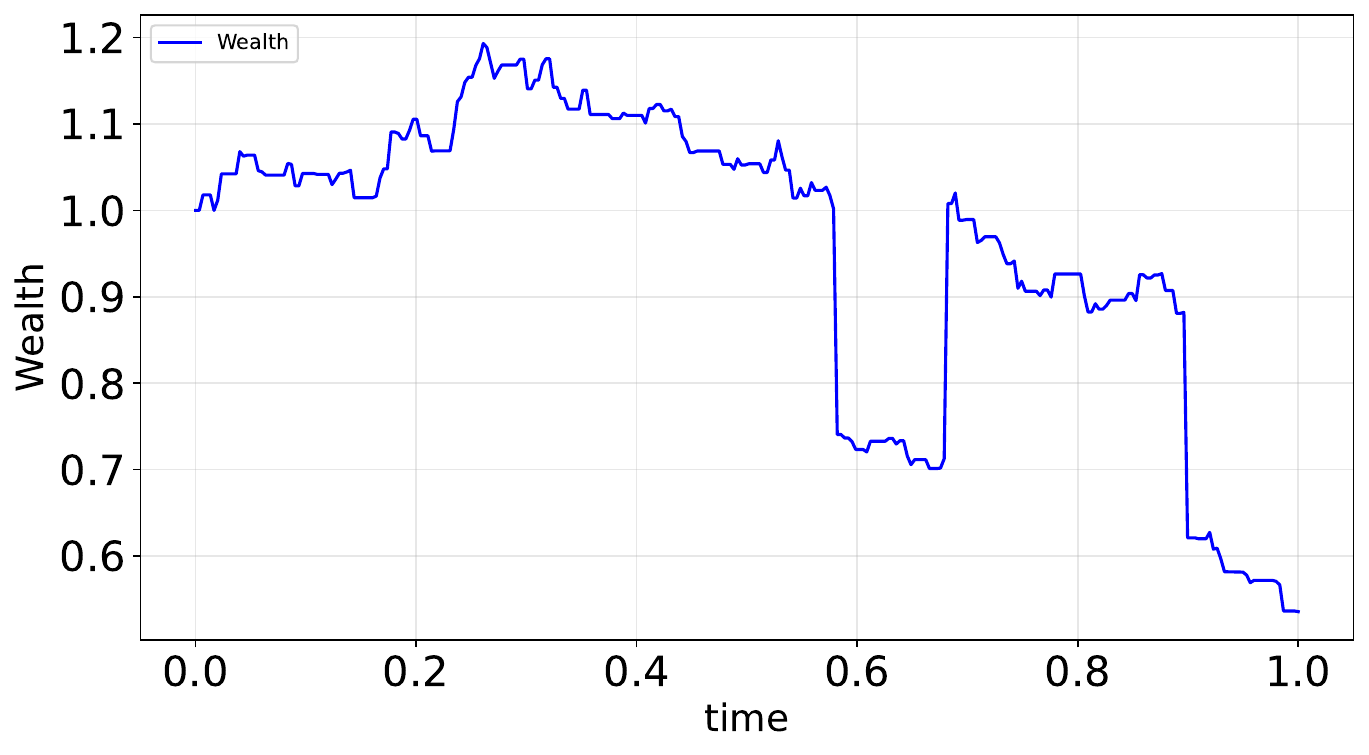}
    \caption{Sample trajectories of the two asset prices $S^1$ (blue) and
$S^2$ (orange), the signal process $G$ (red, right axis), and the
resulting wealth process $X^{\alpha^*}$ under the optimal strategy,
for signal strength $\alpha = 2.0$. The sharp drop in wealth around
$t = 0.6$ coincides with a large negative signal realization,
illustrating the impact of jump risk on portfolio performance.}
    \label{fig:example}
\end{figure}

\begin{figure}[htbp]
    \centering
    \begin{minipage}[b]{0.49\textwidth}
        \centering
        \includegraphics[width=\textwidth]{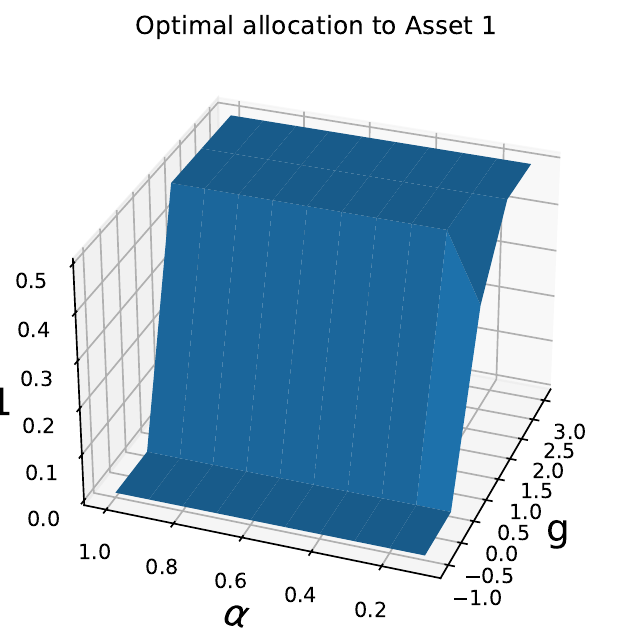}
       
    \end{minipage}
    \hfill
    \begin{minipage}[b]{0.495\textwidth}
        \centering
        \includegraphics[width=\textwidth]{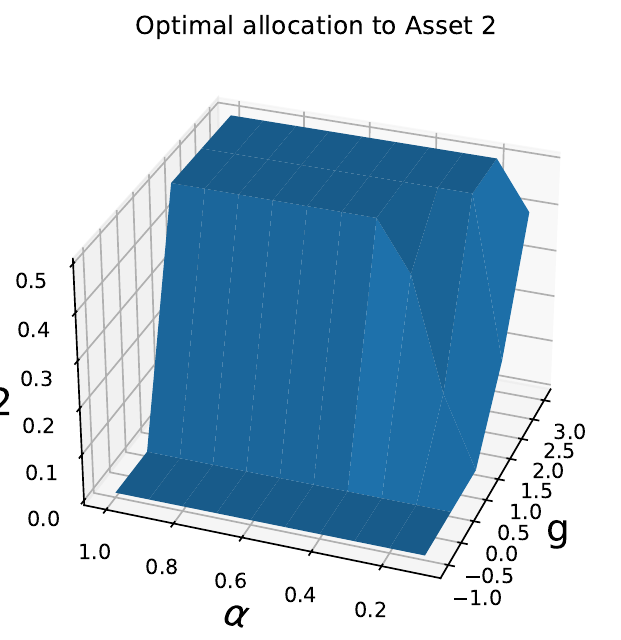}

    \end{minipage}
     \caption{The optimal signal-dependent strategy $(p_1(g), p_2(g))$ as a
function of the signal value $g$ and
the signal strength $\alpha\in[0,1]$. The flat regions correspond to
the credit-line constraints $p^i = \pm\overline\pi$ being active.}
\label{fig:optimal_strategy}
\end{figure}

\begin{figure}[htbp]
    \centering
    \includegraphics[width=0.7\textwidth]{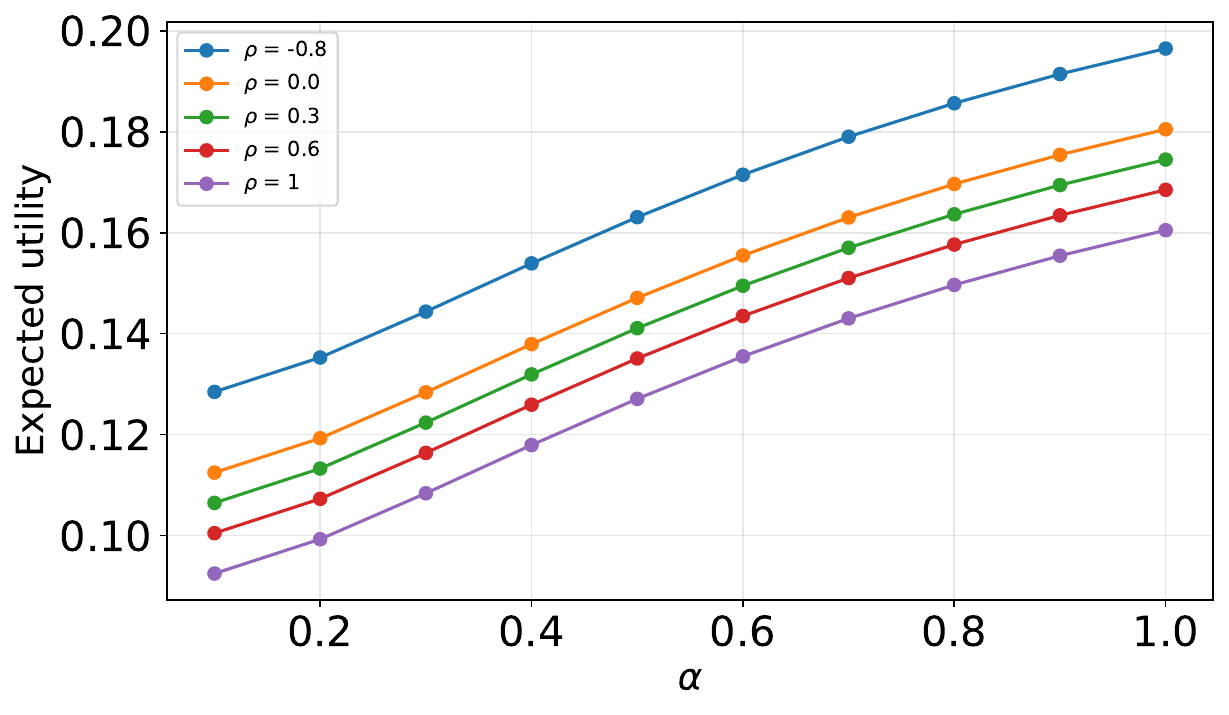}
    \caption{The expected utility of the investor in dependence of the strength of the signal $\alpha$.}
    \label{fig:example}
\end{figure}
\FloatBarrier
\clearpage

\section*{Conclusion}

In this paper, we studied the utility maximization problem in a multidimensional jump--diffusion framework with jump signals and correlated Brownian motions using a BSDE approach. We considered the exponential utility case, extending the one-dimensional framework of \cite{turki2026portfolio}, as well as the power and logarithmic utility cases under the assumption that all assets share a common signal. In the case of asset-specific (idiosyncratic) signals, we showed for the exponential utility that the problem can be reduced to the common-signal setting through a suitable transformation of the mark space and the compensator. The same reduction applies to the power and logarithmic utility cases and is omitted for the sake of brevity.

Several directions for future research naturally arise from this work. In particular, it would be of interest to:
\begin{enumerate}[(i)]
    \item extend the framework to a \emph{multi-agent} setting, leading to multidimensional quadratic BSDEs with jumps. Such systems involve interactions between agents and pose significant analytical challenges beyond the scope of the present paper;

    \item investigate the case of an \emph{infinite} L\'evy measure $\nu$, which requires a more delicate semimartingale decomposition, following the approach of \cite{bank2022merton}.
\end{enumerate}

\textbf{Acknowledgements}

The author would like to express his sincere gratitude to I.~Kharroubi and L.~Abbas-Turki for their invaluable guidance, insightful comments, and many fruitful discussions.

\clearpage
\addcontentsline{toc}{section}{References}
\bibliographystyle{plain}
\bibliography{references}

\begin{thebibliography}{10}

\bibitem{aliprantis2006infinite}
Charalambos~D Aliprantis and Kim~C Border.
\newblock {\em Infinite dimensional analysis: a hitchhiker’s guide}.
\newblock Springer, 2006.

\bibitem{bank2022merton}
Peter Bank and Laura K{\"o}rber.
\newblock Merton's optimal investment problem with jump signals, 2022.

\bibitem{BS25}
Peter Bank and Gemma Sedrakjan.
\newblock How much should we care about what others know? jump signals in
  optimal investment under relative performance concerns.
\newblock {\em arXiv:2503.16039v2}, 2025.

\bibitem{hu2005utility}
Ying Hu, Peter Imkeller, and Matthias M{\"u}ller.
\newblock Utility maximization in incomplete markets.
\newblock {\em Ann. Appl. Probab.}, 15(3):1691--1712, 2005.

\bibitem{jacod2013limit}
Jean Jacod and Albert Shiryaev.
\newblock {\em Limit theorems for stochastic processes}, volume 288.
\newblock Springer Science \& Business Media, 2013.

\bibitem{kazamaki2006continuous}
Norihiko Kazamaki.
\newblock {\em Continuous exponential martingales and BMO}.
\newblock Springer, 2006.

\bibitem{kobylanski}
Magdalena Kobylanski.
\newblock Backward stochastic differential equations and partial differential
  equations with quadratic growth.
\newblock {\em The annals of probability}, 28(2):558--602, 2000.

\bibitem{delong13}
Delong Lukasz.
\newblock {\em {Backward} {Stochastic} {D}ifferenrtial {E}quations with {Jumps}
  and {Their} {Actuarial} and {Financial} {Applications}}.
\newblock Springer-Verlag, 2013.

\bibitem{Merton69}
Robert~C. Merton.
\newblock {Lifetime Portfolio Selection under Uncertainty: The Continuous-Time
  Case}.
\newblock {\em Journal of Economic Theory}, 3(4):373–413, 1971.

\bibitem{Merton71}
Robert~C. Merton.
\newblock Optimum consumption and portfolio rules in a continuous-time model.
\newblock {\em The Review of Economics and Statistics}, 51(3):373–413,
  247-257.

\bibitem{MAM08}
Marie-Am{\'e}lie Morlais.
\newblock Utility maximization in a jump market model.
\newblock {\em Stochastics and Stochastics Reports}, 80:1--27, 2008.

\bibitem{MAM09}
Marie-Am{\'e}lie Morlais.
\newblock Quadratic bsdes driven by a continuous martingale and application to
  utility maximization problem.
\newblock {\em Finance and Stochastics}, 13:121--150, 2009.

\bibitem{oksendal2019applied}
Bernt {\O}ksendal and Agnes Sulem.
\newblock {\em Applied stochastic control of jump diffusions}, volume~3.
\newblock Springer, 2019.

\bibitem{KTPZ15}
Dylan Possamai, Nabil Kazi-Tani, and Chao Zhou.
\newblock {Quadratic BSDEs with jumps: a fixed-point approach}.
\newblock {\em Electron. J. Probab.}, 20:1--28, 2015.

\bibitem{rouge2000pricing}
Richard Rouge and Nicole El~Karoui.
\newblock Pricing via utility maximization and entropy.
\newblock {\em Mathematical Finance}, 10(2):259--276, 2000.

\bibitem{Royer06}
Manuella Royer.
\newblock Backward stochastic differential equations with jumps and related
  non-linear expectations.
\newblock {\em Stochastic Processes and Their Applications},
  116(10):1358--1376, 2006.

\bibitem{turki2026portfolio}
Lokmane~Abbas Turki, Sigui~Brice Dro, and Idris Kharroubi.
\newblock Portfolio exponential utility maximization with jump signals.
\newblock 2026.

\end{thebibliography}

\end{document}